\documentclass[letterpaper,11pt,leqno]{article}
\usepackage[T1]{fontenc}
\usepackage{newtxtext}
\RequirePackage[scaled=.9]{helvet}
\usepackage[dvipsnames,svgnames,x11names]{xcolor}
\usepackage{orcidlink}

\usepackage[authoryear]{natbib}     
\usepackage{hyperref}
\hypersetup{bookmarksnumbered=true, colorlinks=true,
    citecolor=ForestGreen,       
    linkcolor=Blue, 		     
    urlcolor=MidnightBlue     	 
}

\usepackage{amsthm,amsmath,amsfonts,amssymb}
\usepackage{physics,mathrsfs,bm,bbm}
\usepackage{enumitem,multicol,multirow}
\usepackage{graphicx,subfig,tikz}
\graphicspath{{../figure/}{../art/}}

\usepackage[linesnumbered,ruled]{algorithm2e}
\SetKwRepeat{Do}{do}{while}

\theoremstyle{plain}
\newtheorem{theorem}{Theorem}
\newtheorem{lemma}{Lemma}
\newtheorem{proposition}{Proposition}

\theoremstyle{remark}

\newtheorem{example}{Example}
\theoremstyle{definition}

\newtheorem{assumption}{Assumption}

\usepackage{etoolbox}
\newcommand{\renewAformat}[1]{%
  \expandafter\renewcommand\csname the#1\endcsname{A\arabic{#1}}%
  \setcounter{#1}{0}%
}
\newcommand{\renewSformat}[1]{%
  \expandafter\renewcommand\csname the#1\endcsname{S\arabic{#1}}%
  \setcounter{#1}{0}%
}

\newcommand{\R}{\mathbb{R}}					
\newcommand{\N}{\mathsf{N}}                   
\newcommand{\1}{\mathbbm{1}}				
\renewcommand{\Pr}{\mathbb{P}}				
\newcommand{\E}{\mathbb{E}}					
\newcommand{\Cov}{\operatorname{Cov}}		
\newcommand{\normop}[2][]{#1\lVert #2 #1\rVert_{\mathrm{op}}}
\newcommand{\inprod}[3][]{#1\langle #2, #3 #1\rangle}

\usepackage{authblk}
\title{Statistical Inference on Gradient Flows}
\author{Tongyu Li\orcidlink{0009-0007-5051-2716}}
\author{Alexander Giessing\orcidlink{0000-0002-6917-0652}}
\affil{Department of Statistics and Data Science, National University of Singapore\\ \texttt{tongyuli@nus.edu.sg, giessing@nus.edu.sg}}
\date{}

\begin{document}
\maketitle

\begin{abstract}
Gradient-based algorithms are central to modern statistical estimation, yet their statistical analysis is often restricted to fixed-time behavior, such as convergence to a population target or fluctuations at a prescribed iteration. 
In many applications, however, uncertainty quantification is needed along the entire optimization path, especially when the stopping time is data-dependent or divergent.
In this paper, we develop a theory for time-uniform statistical inference on gradient flows arising from empirical risk minimization.
We prove a uniform central limit theorem that characterizes the deviation between empirical and population gradient flows as a continuous-time Gaussian process over the entire nonnegative real line. 
Building on this result, we introduce an algorithm-aware covariance estimator that evolves jointly with the gradient flow and avoids matrix inversion, resampling, or sample splitting.
We show that the covariance estimator is uniformly consistent over time and use it to construct confidence intervals for the target parameter with asymptotically valid coverage.
Our results connect optimization dynamics with statistical inference and provide practical tools for uncertainty quantification in gradient-based methods.
\end{abstract}

\section{Introduction}
\label{sec:intro}
\subsection{Background and motivation}
M-estimation or empirical risk minimization is a foundational approach to constructing estimators from data \citep{vdgeer2000empirical,koltchinskii2011oracle,vdlaan2026researcher}. 
In contemporary applications, however, such estimators are rarely available in closed form. Instead, they are realized as the output of computational procedures, most prominently gradient-based iterative algorithms \citep{bottou2018optimization,lan2020first,wright2022optimization}. 
Consequently, the statistical analysis of algorithmic estimators must go beyond classical asymptotic theory for exact minimizers and explicitly account for the dynamics of the optimization algorithm. 

This perspective has motivated a growing literature at the interface of optimization and statistics that studies the probabilistic behavior of training trajectories. Among recent advances, \citet{agarwal2012fast,chandrasekher2023sharp,loh2017statistical,balakrishnan2017statistical,dwivedi2020singularity} established the convergence rates of parameter estimation, \citet{celentano2021high,han2025entrywise,celentano2025state,gerbelot2024rigorous,benarous2024high} characterized the state evolution that tracks low-dimensional functionals, and \citet{han2025precise,chen2025learning,yan2025semiparametric,dandapanthula2025gradient,martin2026high,fan2026high,nishiyama2026high} developed various model-specific theories to capture finer aspects of training behavior. Complementary lines of research incorporate additional practical constraints, including dependent data \citep{shen2026sgd,liu2023statistical}, contaminated distribution \citep{prasad2020robust,zhou2026minimax}, infinite variance \citep{blanchet2024limit,blanchet2026statistical}, memory limitations \citep{berg2024statistical,quan2024optimal}, and privacy requirements \citep{avella2023differentially,xia2025statistical}. 

Building on these developments, a growing body of work seeks to tailor statistical inference procedures to algorithmic estimators, with the goal of quantifying uncertainty directly from the training trajectories. 
Early contributions have focused on one-pass stochastic gradient algorithms and established asymptotic normality of the iterates or their averaged versions \citep{kushner2003stochastic}, together with practical methods for covariance estimation. 
For example, \citet{chen2020statistical} proposed plug-in and batch-means estimators for the asymptotic covariance of the averaged stochastic gradient descent, allowing the construction of confidence intervals from the algorithmic path alone.
Subsequent works introduced more computationally efficient inference procedures for stochastic optimization, including Kiefer--Wolfowitz methods \citep{chen2024online}, recursive score-based estimators for high-dimensional generalized linear models \citep{shi2021statistical}, and online covariance estimation techniques that track the evolving variability of iterates \citep{zhu2023online}. 
More recently, \citet{han2024online} extended these ideas to high-dimensional and debiased settings, \citet{sheshukova2025gaussian,butyrin2025improved} derived Gaussian approximations and bootstrap procedures, and \citet{carter2025statistical} proposed inference frameworks applicable to a broader class of online algorithms. 
In parallel, other work has investigated more general gradient-based schemes, including uncertainty quantification for early-stopped iterative estimators in linear models \citep{bellec2024uncertainty} and gradient-based debiased inference in single-index regression under the mean-field regime \citep{han2024gradient}.

Despite this progress, most existing results concern the statistical analysis at a fixed iteration or the terminal estimator, leaving the behavior of the entire training trajectory comparatively less understood. 
In practice, algorithmic dynamics often evolve over long time horizons and the stopping rule is data-dependent. Hence, valid uncertainty quantification must control the estimator uniformly along the path rather than at a single time point. Without such time-uniform control, inference at a chosen iteration may fail to reflect the accumulated stochastic fluctuations of the algorithm, and coverage guarantees can break down when the number of iterations is random or diverges. 
Recent works have begun to explore aspects of this problem, including universality and long-time dynamics of gradient descent in ultra-high dimensions \citep{han2025long}, time-uniform concentration inequalities for iterative algorithms \citep{xie2024asymptotic,pham2025time,kar2026high}, and long-time or functional central limit theorems for stochastic approximation schemes \citep{agrawalla2025statistical,flamand2026functional}. 
Nevertheless, a general theory of time-uniform statistical inference along the training trajectory remains largely underdeveloped and has emerged as a fundamental open problem  \citep[e.g.,][Section~5]{bellec2024uncertainty}.

\subsection{Our contributions}
We develop a time-uniform theory for statistical inference on gradient flows, including distributional approximations and covariance estimation. 
Gradient flow is the continuous-time analogue of full-batch gradient descent, and therefore provides a tractable model for studying the statistical behavior of algorithmic estimators along their trajectories.
The contributions of our work are threefold:

\begin{itemize}
\item In Section~\ref{sec:converge}, we prove a \textbf{uniform central limit theorem for gradient flows}, thereby bridging optimization dynamics and empirical process theory. 
The starting point is a comparison between the empirical gradient flow with its population counterpart: after linearization, their discrepancy can be represented as an empirical process indexed by time-parametrized functions (Section~\ref{sec:linear}). 
Thus, the uniform central limit theorem over the full time domain reduces to establishing the Donsker property of a suitable one-parameter function class. 
The main challenge is that its index set, although one-dimensional, is unbounded. 
To control the complexity of one-parameter function classes, we establish an arc length-based condition (Appendix~\ref{appn:Donsker}) and verify the corresponding finiteness for the sensitivity functions along the gradient flow (Section~\ref{sec:lead}). 
By combining an integral representation of the process with a bootstrap technique for differential equations, we obtain a time-uniform treatment of the infinite time domain (Section~\ref{sec:main}).  In this way, we extend classical pointwise asymptotic theory to a pathwise setting. 
This shows that gradient-flow trajectories, despite evolving over an infinite interval, trace low-complexity subsets of the ambient function space and hence admit uniform Gaussian asymptotics. 

\item In Section~\ref{sec:cov_est}, we develop a practical and theoretically grounded method for \textbf{uncertainty quantification along the entire optimization path}.
Specifically, we propose an algorithm-aware covariance estimation procedure  that couples inference with the evolution of the gradient flow (Section~\ref{sec:method}). 
The estimator is constructed through an auxiliary dynamical system that captures the first-order sensitivity of the trajectory to perturbations in the empirical measure. It can be implemented alongside the primary algorithm with moderate additional cost and, unlike classical approaches, avoids matrix inversion, resampling, and sample splitting. 
We show that this estimator converges uniformly to the population covariance function at a nearly parametric rate (Section~\ref{sec:theory}). Combined with the uniform distributional Gaussian approximation, this enables the construction of asymptotically valid confidence intervals for the target parameter and accommodates data-dependent and diverging stopping times.
To the best of our knowledge, this is the first inference method for gradient-based estimators beyond fixed-time asymptotics. 

\item In Section~\ref{sec:numerical}, we complement our theoretical developments with numerical studies that \textbf{illustrate the practical implementation and validate the finite-sample performance}.
On the algorithmic side (Section~\ref{sec:algo}), we provide concrete iterative procedures for solving the coupled dynamics underlying gradient flows and covariance estimation, based on standard discretization schemes such as Euler and higher-order Runge–Kutta methods. 
On the empirical side (Section~\ref{sec:simu}), simulation experiments across a range of statistical models demonstrate that the proposed time-uniform inference method achieves accurate coverage and stable behavior. 
The results also confirm that the Gaussian asymptotics derived from the the uniform central theorem for gradient flows provide a reliable basis for uncertainty quantification of gradient descent and that the performance is robust to the choice of the numerical solver.
These findings substantiate the practical relevance of our theory and highlight the effectiveness of integrating optimization dynamics with statistical inference. 
\end{itemize}


\subsection{Notation}
Let $\nu f = \int f \dd{\nu}$ for any signed measure $\nu$ and integrable function $f$, where $f$ is allowed to take values as vectors or matrices. 
The real line $\R$ and its subsets are equipped with the Lebesgue measure. 
Write $\norm{\cdot}_{\ell^2}$ for the Euclidean norm, i.e., $\norm{v}_{\ell^2} = (\sum_{k=1}^{d}v_k^2)^{1/2}$ for $v = (v_1,\dots,v_d)^\top \in \R^d$. 
If $\nu$ is a measure and $p \ge 1$, then the $L^p(\nu;\R^d)$-norm is defined for $\R^d$-valued measurable functions $f$ by $\norm{f}_{L^p(\nu)} = (\nu \norm{f}_{\ell^2}^p)^{1/p}$, where $\norm{f}_{L^\infty(\nu)}$ is the $\nu$-essential supremum of $\norm{f}_{\ell^2}$. Denote $L^p(\nu) = L^p(\nu;\R)$ for simplicity. 
Let $\normop{A} = \sup_{v:\norm{v}_{\ell^2}\le1} \norm{Av}_{\ell^2}$ be the operator norm of a matrix $A$, calculated as the largest singular value of $A$. 
Denote by $0_d$ the zero vector in $\R^d$, by $e_k = (0_{k-1}^\top,1,0_{d-k}^\top)^\top$ the $k$th standard basis vector of $\R^d$, and by $\1\{\cdot\}$ the indicator function. 
For a mapping $f$ between normed vector spaces, $\nabla f$ denotes its Fr\'echet derivative, which coincides with the Jacobian matrix in the case where $f$ is a vector-valued function of several variables. 
For two sequences $(X_n)$ and $(Y_n)$ of nonnegative random variables, $Y_n = o_\Pr(X_n)$ means that $\lim_{\delta\searrow0} \limsup_{n\to\infty} \Pr(Y_n > \delta X_n) = 0$, and $Y_n = \mathcal{O}_\Pr(X_n)$ means that $\lim_{C\to\infty} \limsup_{n\to\infty} \Pr(Y_n > C X_n) = 0$.  We denote convergence in distribution of a sequence of random variables $(X_n)$ to a random variable $Z$ by $X_n \xrightarrow{d}Z$.

\section{Problem setup}
\label{sec:setup}
Let $Z_1,\dots,Z_n$ be independent observations drawn from a common distribution $P$, and denote by $P_n$ the corresponding empirical measure. For a family of criterion functions $m_\theta$ indexed by $\theta\in\R^d$, define the population and empirical risk functions 
\[ M(\theta) = P m_\theta = \E\{m_\theta(Z_1)\} \qand 
M_n(\theta) = P_n m_\theta = n^{-1}\sum_{i=1}^{n} m_\theta(Z_i) .\]
The target parameter $\theta_*\in\R^d$ is a minimizer of the population risk $M$, and its natural sample analogue is an empirical minimizer of $M_n$. In many modern statistical problems, the minimizer is not available in closed from and must be obtained numerically by an optimization algorithm. 

We model this optimization algorithm through gradient-flow dynamics. Gradient flow is the continuous-time analogue of full-batch gradient descent
\citep{scieur2017integration,francca2021gradient} and has been widely used to study optimization trajectories, e.g., in neural-network models \citep{zhao2023symmetries,wendin2025gradient}.
Let $\psi_\theta = \pdv*{m_\theta}{\theta}$, interpreted as a subgradient when the criterion is non-smooth. The subgradient of $M_n$ is then
\[\Psi_n(\theta) = P_n \psi_\theta = n^{-1}\sum_{i=1}^{n} \psi_\theta(Z_i).\]
We study the functional estimator $\hat{\theta}$ defined through the gradient flow equation 
\begin{equation}\label{eq:gradflow_emp}
\dv{\hat{\theta}(t)}{t} = -\Psi_n(\hat{\theta}(t)),\qquad \hat{\theta}(0) = \theta_0,
\end{equation}
where $\theta_0 \in \R^d$ is a fixed initialization. 
Our focus on full-batch dynamics is motivated in part by settings in which full-batch gradient methods can outperform one-pass stochastic variants \citep{wu2025risk,kovavcevic2026full}.

Under suitable regularity conditions on $M_n$, the empirical gradient flow converges as $t\to\infty$ to a terminal estimator, denoted by $\hat{\theta}(\infty)$. This terminal point is an M-estimator.
The classical analyses of iterative algorithms then often rely on the terminal decomposition 
\[ \hat{\theta}(t) - \theta_* = \{\hat{\theta}(t) - \hat{\theta}(\infty)\} + \{\hat{\theta}(\infty) - \theta_*\} .\]
The first term is an optimization gap, made negligible by running the algorithm long enough. The second term is the empirical fluctuation and is governed by classical asymptotic theory. 
This perspective is well suited to exact or near-exact empirical minimizers, but it is less suited to early-stopped or data-adaptively stopped algorithms. Indeed, it can be difficult to make the optimization error uniformly negligible over random empirical landscapes. Moreover, inference based solely on the terminal estimator $\hat{\theta}(\infty)$ discards the statistical information contained in the algorithmic trajectory.

We therefore take a different perspective and study the stochastic fluctuations along the entire optimization path. 
Putting \[\Psi(\theta) = P\psi_\theta = \E\{\psi_\theta(Z_1)\},\] we introduce the gradient flow $\theta^\circ$ at the population level, defined as the solution of 
\begin{equation}\label{eq:gradflow_pop}
\dv{\theta^\circ(t)}{t} = -\Psi(\theta^\circ(t)), \qquad \theta^\circ(0) = \theta_0.
\end{equation}
Under regularity conditions on $M$, the limit of the population gradient flow $\theta^\circ(t)$ converges as $t \to \infty$ to the target parameter $\theta_*$. 
We now replace the terminal comparison by the pathwise decomposition
\[ \hat{\theta}(t) - \theta_* = \{\hat{\theta}(t) - \theta^\circ(t)\} + \{\theta^\circ(t) - \theta_*\} .\]
The first term is the stochastic fluctuation of the empirical trajectory around its population counterpart; the second term is a deterministic population bias incurred by running the optimization.
Both terms depend on $t$, but only the first carries sampling noise. This decomposition keeps the optimization path in the statistical analysis and separates sampling fluctuation from population-level early-stopping bias.

The main object of study in this paper is the fluctuation process $\{n^{1/2}(\hat{\theta}-\theta^\circ)(t):t\in[0,\infty)\}$, for which we seek weak convergence and covariance estimation uniformly over $[0,\infty)$. 
Since the analysis concerns entire trajectories, we assume throughout that the empirical and population gradient flows are well-defined on the nonnegative real line. This is automatic, for instance, when the driving vector fields are globally Lipschitz. More generally, let $\Theta$ be a Hilbert space equipped with norm $\norm{\cdot}$. For a driving function $g : \Theta \to \Theta$ and an initial point $\theta_0 \in \Theta$, consider the initial value problem 
\[ \dv{\theta(t)}{t} = -g(\theta(t)) , \quad \theta(0)=\theta_0 .\]
Even if $g$ is discontinuous, as may occur for nonsmooth losses such as quantile regression, the Br\'ezis--Komura theorem \citep[Theorem~11.7]{ambrosio2024lectures} still yields existence and uniqueness of the gradient flow. 
The corresponding sufficient condition is that $g$ be a sub-gradient of a lower semicontinuous function $f : \Theta \to \R$ such that $f - (\lambda/2)\norm{\cdot}^2$ is convex for some $\lambda\in\R$, not necessarily nonnegative.

As a first step toward the time-uniform analysis, the following elementary comparison bound shows that the deviation between the empirical and population flows is controlled by the empirical process indexed by the population trajectory.
\begin{proposition}\label{prop:vanilla}
For any $\lambda\in\R$, it holds on the event $\{M_n - (\lambda/2)\norm{\cdot}_{\ell^2}^2 \text{ is convex}\}$ that 
\[\begin{aligned}
\norm{\hat{\theta}(t)-\theta^\circ(t)}_{\ell^2} 
&\le \int_0^t \exp\{-\lambda(t-s)\} \norm{\Psi_n(\theta^\circ(s))-\Psi(\theta^\circ(s))}_{\ell^2} \dd{s} \\
&\le t \exp(\lambda^{-}t) \sup_{s\in[0,t]}\norm{(P_n-P)\psi_{\theta^\circ(s)}}_{\ell^2} ,\quad \forall t\in[0,\infty) ,
\end{aligned}\]
where $\lambda^- = \max(-\lambda,0)$ is the negative part of $\lambda$.
\end{proposition}

Proposition~\ref{prop:vanilla} gives a baseline stability comparison on each fixed horizon $t<\infty$. It yields consistency uniformly over $s\in[0,t]$ whenever the class $\{e_k^\top\psi_{\theta^\circ(s)}:s\in[0,t],\,1\le k\le d\}$ is $P$-Glivenko--Cantelli, which follows under standard regularity conditions from \citet[Theorems~2.4.1 and~2.7.17]{vVW2023weak}. 
This is only a compact-time law-of-large-numbers result: it does not imply tightness, a Donsker theorem, or weak convergence of the stochastic trajectory $t\mapsto\hat\theta(t)-\theta^\circ(t)$. 
Nor does it provide uniform control over $t\in[0,\infty)$. The latter requires additional curvature or sharper empirical-process arguments. 
If $M_n$ is strongly convex, the exponential kernel is integrable and the same comparison can yield full time-uniform consistency. Without sufficient positive curvature, the prefactor $t\exp(\lambda^-t)$ diverges as $t\to\infty$, motivating the sharper time-uniform analysis developed in the next section.

Before turning to the time-uniform Gaussian asymptotics, we record several standard examples of criteria to which the gradient-flow framework applies.
\begin{example}[Generalized linear regression]\label{ex:glr}
Suppose that observations are made on several predictor variables and one response variable, so $Z_i = (X_i,Y_i)$, $i=1,\dots,n$, take values in $\R^d \times \mathcal{Y}$. 
In many regression applications, the criterion function has the form \[ m_\theta(x,y) = l(x^\top\theta,y) + p(\theta) \] for some functions $l : \R\times\mathcal{Y} \to \R$ and $p : \R^d \to [0,\infty)$ that assess the fidelity to data and the plausibility of parameter, respectively. 
Hence, \[\psi_\theta(x,y) = l'(x^\top\theta,y)x + \pdv*{p(\theta)}{\theta},\] where $l'$ is the subgradient of $l$ with respect to the first argument. 
The following is a non-exhaustive list of choices of $l$ for different purposes: 
\begin{itemize}
\item linear regression: $l_{\mathrm{lin}}(a,y) = (y-a)^2/2$ with $\mathcal{Y}=\R$;
\item logistic regression: $l_{\mathrm{logit}}(a,y) = \log(1+\exp(a)) - a y$ with $\mathcal{Y}=\{0,1\}$;
\item Poisson regression: $l_{\mathrm{Poi}}(a,y) = \exp(a) - a y$ with $\mathcal{Y}=\mathbb{N}$;
\item $\tau$th quantile regression:  $l_{\mathrm{qt}}(a,y) = (y-a)(\tau-\1\{y<a\})$ with $\mathcal{Y}=\R$ and $\tau\in(0,1)$; 
\item phase retrieval: $l_{\mathrm{pr}}(a,y) = (y-a^2)^2/2$ with $\mathcal{Y}=\R$. 
\end{itemize}
The Br\'ezis--Komura theorem applies directly to the convex examples above with convex penalty. Indeed, $\theta \mapsto l(x^\top\theta,y)$ inherits the convexity from $a \mapsto l(a,y)$ in linear/ logistic/ Poisson/ quantile regression. 
The phase-retrieval loss is non-convex, but the smallest eigenvalue of its Hessian is universally lower bounded by $-2\normop{yxx^\top}$. Indeed, we have the lower bound $\pdv*[2]{l_{\mathrm{pr}}(a,y)}{a} = 6a^2 - 2y \ge - 2y$. 
\end{example}

Among these examples, ordinary least squares is especially useful because both the empirical and population gradient flows are available in closed form. The next example uses this explicit structure to derive a uniform linear expansion for the pathwise fluctuation.

\begin{example}[Ordinary least squares]\label{ex:ols}
Consider the case of Example~\ref{ex:glr} where the squared error loss for linear regression is applied with zero penalty, i.e., $m_\theta(x,y) = (y-x^\top\theta)^2/2$ and $p \equiv 0$. 
Then the gradient flows \eqref{eq:gradflow_emp} and \eqref{eq:gradflow_pop} admit the closed-form expressions
\[ \hat{\theta}(t) = \exp\!\big(\!-t\hat{\Sigma}\big)\theta_0 + \int_0^t \exp\!\big(\!-s\hat{\Sigma}\big) \hat{\xi} \dd{s} ,\]
\[ \theta^\circ(t) = \exp(-t\Sigma)\theta_0 + \int_0^t \exp(-s\Sigma) \xi \dd{s} ,\]
where $\hat{\Sigma} = n^{-1}\sum_{i=1}^{n}X_iX_i^\top$, $\hat{\xi} = n^{-1}\sum_{i=1}^{n}Y_iX_i$, $\Sigma = \E(X_1X_1^\top)$ and $\xi = \E(Y_1X_1)$. 
Let $\tilde{\Sigma}_\alpha = \alpha\hat{\Sigma} + (1-\alpha)\Sigma$ and $\tilde{\xi}_\alpha = \alpha\hat{\xi} + (1-\alpha)\xi$ and write
\[\begin{aligned}
& (\hat{\theta}-\theta^\circ)(t) = \int_0^1 \pdv{\alpha}\bigg\{\exp\!\big(\!-t\tilde{\Sigma}_\alpha\big)\theta_0 + \int_0^t \exp\!\big(\!-s\tilde{\Sigma}_\alpha\big) \tilde{\xi}_\alpha \dd{s}\bigg\} \dd{\alpha} \\
&= \int_0^1 \big\{-F_0(t;\tilde{\Sigma}_\alpha)[\hat{\Sigma}-\Sigma]\theta_0 -F_1(t;\tilde{\Sigma}_\alpha)[\hat{\Sigma}-\Sigma]\tilde{\xi}_\alpha +F_2(t;\tilde{\Sigma}_\alpha)(\hat{\xi}-\xi)\big\} \dd{\alpha} ,
\end{aligned}\]
where $F_0,F_1,F_2$ are symmetric matrix‑parametrized functions defined as 
\[ F_0(t;A)[E] = \int_0^t \exp\{-(t-s)A\} E \exp(-sA) \dd{s} ,\]
\[ F_1(t;A)[E] = \iint_{0\le u\le s\le t} \exp\{-(t-s)A\} E \exp(-uA) \dd{u}\dd{s} ,\]
and $F_2(t;A) = \int_0^t \exp(-sA) \dd{s}$. 
Assume that the smallest eigenvalue of $\Sigma$ is $\lambda_* > 0$. Then, the perturbation bound \citep[Corollary~4.1]{kaagstrom1977bounds}
\[ \normop{\exp\!\big(\!-t\tilde{\Sigma}_\alpha\big)-\exp(-t\Sigma)} \le t \alpha\normop{\hat{\Sigma}-\Sigma} \exp\{-(\lambda_*-\alpha\normop{\hat{\Sigma}-\Sigma})t\} \]
implies that on the event $\{\normop{\hat{\Sigma}-\Sigma}<\lambda_*\}$, 
\[ \sup_{\alpha\in[0,1]}\sup_{t\in[0,\infty)} \normop{F_j(t;\tilde{\Sigma}_\alpha)-F_j(t;\Sigma)} \le C \normop{\hat{\Sigma}-\Sigma} / (\lambda_*-\normop{\hat{\Sigma}-\Sigma})^{3-\abs{j-1}} \]
for $j=0,1,2$ and some universal constant $C>0$. 
In view of the finite-dimensional central limit theorem for $(\hat{\Sigma},\hat{\xi})$, we can conclude that 
\[ n^{1/2}(\hat{\theta}-\theta^\circ) \approx -F_0(\cdot;\Sigma)[n^{1/2}(\hat{\Sigma}-\Sigma)]\theta_0 -F_1(\cdot;\Sigma)[n^{1/2}(\hat{\Sigma}-\Sigma)]\xi +F_2(\cdot;\Sigma)n^{1/2}(\hat{\xi}-\xi) \]
converges in distribution to a zero-mean Gaussian process uniformly on $[0,\infty)$.
\end{example}
The ordinary least squares example suggests the general mechanism developed next: the empirical gradient flow admits a linearized leading term that can be written as an empirical process indexed by the population trajectory. The main task is to prove a Gaussian asymptotics for this time-indexed empirical process uniformly over $[0,\infty)$.


\section{Time-uniform central limit theorem}
\label{sec:converge}
In this section, we analyze the fluctuation process
$\hat{\theta} - \theta^\circ$
over the entire interval $[0,\infty)$, laying the foundation for time-uniform statistical inference.
The key insight is that the population gradient flow trajectory gives rise to the time-indexed class of sensitivity functions $\{\Phi_t : t \geq 0\}$ and that this class can have low complexity even though the index set $[0,\infty)$ is unbounded. 

\subsection{Linearization of gradient flows}
\label{sec:linear}

In statistical asymptotic theory \citep{serfling1980approximation,hall2023course}, linearization means approximating the estimation error by a linear functional of the empirical process, usually through the estimator's influence function at the population truth. This is distinct from lazy training in optimization \citep{chizat2019lazy}, which linearizes the model with respect to parameter displacement around the initialization. Here, the relevant notion of linearity is instead linearity in the sampling distribution, as required for asymptotic statistical inference.

Throughout this paper, we work in the differentiable regime and assume that 
\[H(\theta) = \nabla \Psi(\theta)\]
exists along the population trajectory. Under regularity conditions on the population risk $M$, $\Psi = \nabla M$ and thus $H$ is the Hessian of $M$. 
To describe the evolution of gradient flows at the population level, let the matrix $\Pi(t,s) \in \R^{d\times d}$, $s,t\in[0,\infty)$, be defined by the matrix-valued initial value problem 
\begin{equation}\label{eq:transition_matrix}
\pdv{t}\Pi(t,s) = -H(\theta^\circ(t)) \Pi(t,s) ,\quad \Pi(s,s)=I_d ,
\end{equation}
where $I_d = (e_1,\dots,e_d)$ is the identity matrix.
The matrix $\Pi(t,s)$ is called the principal fundamental matrix solution (at $s$). If $\{H(\theta^\circ(t))\}_{t\in[0,\infty)}$ is a family of commuting matrices, then
\[\Pi(t,s) = \exp\{-\int_s^t H(\theta^\circ(u)) \dd{u}\}.\]
Tracking the first-order perturbation of the gradient flows, we define the sensitivity function $\Phi_t$ pointwise for $t\in[0,\infty)$ as the accumulation of propagated gradients,
\begin{equation}\label{eq:func_int}
\Phi_t(\cdot) = \int_0^t \Pi(t,s) \psi_{\theta^\circ(s)}(\cdot) \dd{s}.
\end{equation}
This vector-valued function records how empirical perturbations of the gradient field accumulate along the population trajectory.

\begin{lemma}[Linearization of the gradient flow]\label{lem:decomposition}
For the empirical and population gradient flows $\hat{\theta}$ and $\theta^\circ$ defined by \eqref{eq:gradflow_emp} and \eqref{eq:gradflow_pop}, respectively, we have for every $t\in[0,\infty)$,
\[ \hat{\theta}(t)-\theta^\circ(t) = \Delta_n(t) - \int_0^t \Pi(t,s) \{R_n(s) + D_{\Psi}(\hat{\theta}(s),\theta^\circ(s))\} \dd{s} ,\]
where $\Delta_n , R_n : [0,\infty) \to \R^d$ and $D_\Psi : \R^d\times\R^d \to \R^d$ are given by 
\begin{equation}\label{eq:lead}
\Delta_n(t) = -\int_0^t \Pi(t,s) (\Psi_n-\Psi)(\theta^\circ(s)) \dd{s} = -(P_n-P)\Phi_t ,
\end{equation}
\begin{equation}\label{eq:doubly-centered}
R_n(t) = (\Psi_n-\Psi)(\hat{\theta}(t))-(\Psi_n-\Psi)(\theta^\circ(t)) ,
\end{equation}
\begin{equation}\label{eq:Bregman_div}
D_\Psi(\theta,\beta) = \Psi(\theta) - \Psi(\beta) - H(\beta)(\theta-\beta) .
\end{equation}
\end{lemma}
The term $\Delta_n$ is the leading linear empirical-process term. The two remaing terms are higher-order remainders: $R_n$ is a doubly centered gradient increment, and $D_\Psi$ is the Bregman divergence of the population gradient map $\Psi$ derived as the remainder in a first-order Taylor expansion. 

\begin{example}[Linearization in ordinary least squares] Continue from Example~\ref{ex:ols}. 
We have $\Pi(t,s) = \exp\{-(t-s)\Sigma\}$, and thus \eqref{eq:func_int} is given by 
\[\Phi_t(x,y) = F_0(t;\Sigma)[xx^\top] \theta_0 + F_1(t;\Sigma)[xx^\top] \xi - F_2(t;\Sigma) yx.\]
Meanwhile, the higher-order terms \eqref{eq:doubly-centered} and \eqref{eq:Bregman_div} are 
\[ R_n = (\hat{\Sigma}-\Sigma)(\hat{\theta}-\theta^\circ) \qand D_\Psi \equiv 0_d .\]
Thus, the only nonlinear remainder is the product of the empirical covariance fluctuation $\hat{\Sigma} - \Sigma$ and the flow error $\hat{\theta} - \theta^\circ$, which is second order under the usual $n^{-1/2}$-scaling.
\end{example}

By Lemma~\ref{lem:decomposition},  the asymptotic behavior of $\hat{\theta} - \theta^\circ$ is governed by the empirical process $\Delta_n(t) =  - (P_n - P) \Phi_t$, whenever the two remainder terms are negligible uniformly in time. 
We defer the control of the linearization error to Section~\ref{sec:main}, and first study the uniform convergence of the leading term $\Delta_n$.

\subsection{Uniform central limit theorem for \eqref{eq:lead}}
\label{sec:lead}
The leading term $\Delta_n$ in~\eqref{eq:lead} is an empirical process indexed by the class $\{- \Phi_t: t \geq 0\}$. Thus, the main task is to establish a Donsker property for this time-indexed class. The following lemma gives a simple sufficient condition in terms of the $L^2(P)$-arc length of $t \mapsto - \Phi_t$.
An abstract version is treated in Appendix~\ref{appn:Donsker}. 

\begin{lemma}\label{lem:lead}
Assume that \[\int_0^\infty \norm{\norm{\pdv{\Phi_t}{t}}_{\ell^2}}_{L^2(P)} \dd{t} < \infty.\] Then, 
$n^{1/2}\Delta_n$ converges weakly in $L^\infty([0,\infty); \R^d)$ to a zero-mean Gaussian process with covariance function $G : [0,\infty)\times[0,\infty) \to \R^{d\times d}$ given in terms of \eqref{eq:func_int} by 
\begin{equation}\label{eq:cov}
G(t_1,t_2) = \Cov_P(\Phi_{t_1},\Phi_{t_2}) = P\Phi_{t_1}\Phi_{t_2}^\top - P\Phi_{t_1} P\Phi_{t_2}^\top.
\end{equation}
\end{lemma}

To apply Lemma~\ref{lem:lead}, it remains to verify that the sensitivity path $(\Phi_t)_{t\geq0}$ has finite $L^2(P)$-arc length. We give sufficient conditions in the rest of this section. 
The key observation is that $(\Phi_t)_{t\geq0}$ itself satisfies a linear evolution equation:
\begin{equation}\label{eq:ode-func_int}
\pdv{t}\Phi_t = - H(\theta^\circ(t)) \Phi_t + \psi_{\theta^\circ(t)} ,\quad \Phi_0=0_d.
\end{equation}
Indeed, differentiating \eqref{eq:func_int} and using \eqref{eq:transition_matrix} gives,
\[ \pdv{t}\Phi_t = \Pi(t,t) \psi_{\theta^\circ(t)} + \int_0^t \pdv{t}\Pi(t,s) \psi_{\theta^\circ(s)} \dd{s} = \psi_{\theta^\circ(t)} - H(\theta^\circ(t)) \int_0^t \Pi(t,s) \psi_{\theta^\circ(s)} \dd{s}. \]

The assumptions below control the variability of the population gradient flow $\theta^\circ$ and the sensitivity path $(\Phi_t)_{t\ge0}$. 
They permit certain non-convex landscapes for the criterion function $\theta \mapsto m_\theta$ and do not require pointwise differentiability of $\theta \mapsto \psi_\theta$. Instead, the regularity conditions are placed only on the population gradient field $\Psi(\theta) = P\psi_\theta$, whose differentiability and local smoothness govern the evolution of the population flow. 
This framework accommodates models with non-smooth or non-convex empirical objectives, including phase retrieval and quantile regression in Example~\ref{ex:glr}.

\begin{assumption}\label{asm:converge}
The population gradient flow $\theta^\circ$ in \eqref{eq:gradflow_pop} converges exponentially to a point $\theta_* \in \R^d$ with $\Psi(\theta_*) = 0_d$, i.e., there exist some positive constants $C_0$ and $\mu$ such that 
\[ \norm{\theta^\circ(t)-\theta_*}_{\ell^2} \le C_0 \exp(-\mu t) ,\quad \forall t\in[0,\infty) .\]
\end{assumption}
Exponential convergence of gradient flows and gradient-based algorithms is a classical topic in optimization theory \citep[see, e.g.,][]{dello2024local,weissmann2025almost}. One common sufficient condition is a local gradient dominance condition, also known as the local \citeauthor{polyak1963gradient}--\citeauthor{lojasiewicz1963propriete} inequality: For each 
$\theta_\star \in \R^d$ with $\Psi(\theta_\star)=0_d$, there is a neighborhood $\Theta_\star$ of $\theta_\star$ and a constant $c(\theta_\star)>0$ such that 
\begin{equation}\label{eq:PL}
\norm{\Psi(\theta)}_{\ell^2} \ge c(\theta_\star) \abs{M(\theta)-M(\theta_\star)}^{1/2} ,\quad \forall\theta\in\Theta_\star \subset \R^d .
\end{equation}
If $M - (\lambda/2)\norm{\cdot}_{\ell^2}^2$ is convex for some constant $\lambda > 0$, then \eqref{eq:PL} is satisfied with $c \equiv (2\lambda)^{1/2}$~\citep{karimi2016linear}.
More generally, \eqref{eq:PL} has been shown to hold in diverse non-convex settings, including deep neural networks with analytic activation functions \citep[Remark~2.3]{weissmann2025almost}.
Together with boundedness of $\theta^\circ$, \eqref{eq:PL} implies Assumption~\ref{asm:converge} follows with $\mu = c^2(\theta_*)/2$ from \L{}ojasiewicz's theorem \citep[Theorem~10.1.6 combined with Lemma~2.1.4]{haraux2015convergence}.

We next impose regularity conditions on the gradients and Hessian matrices that arise in the evolution of the gradient flow. Let \[\mathcal{B}(\theta_*,r) = \{\theta\in\R^d : \norm{\theta-\theta_*}_{\ell^2} \le r\},\] denote the Euclidean ball of radius $r\in[0,\infty)$ centered at $\theta_*$.

\begin{assumption}\label{asm:grad-Lip}
The subgradient function $\psi_\theta = \pdv*{m_\theta}{\theta}$ is square integrable in the sense that $\norm{\psi_\theta}_{\ell^2} \in L^2(P)$ for any $\theta\in\R^d$. 
Moreover, for any $r\in[0,\infty)$, there exists an envelop function $\dot{\psi}_r \in L^2(P)$ such that \[\norm{\psi_{\theta_1}-\psi_{\theta_2}}_{\ell^2} \le \norm{\theta_1-\theta_2}_{\ell^2} \dot{\psi}_r ,\quad \forall\theta_1,\theta_2\in \mathcal{B}(\theta_*,r) .\]
\end{assumption}

\begin{assumption}\label{asm:Hess-Lip}
The matrix-valued function $H = \nabla \Psi$ is locally Lipschitz continuous, i.e., for any $r\in[0,\infty)$, there exists some constant $L(r) > 0$ such that \[\normop{H(\theta_1)-H(\theta_2)} \le L(r) \norm{\theta_1-\theta_2}_{\ell^2} ,\quad \forall\theta_1,\theta_2\in \mathcal{B}(\theta_*,r) .\]
\end{assumption}

\begin{assumption}\label{asm:coercive}
Let $\lambda^\circ(t)$ be the smallest eigenvalue of $H(\theta^\circ(t))$. 
There exist constants $\lambda_* > 0$ and $t_* \in [0,\infty)$ such that \[\lambda^\circ(t) \ge \lambda_* ,\quad \forall t \in [t_*, \infty) .\]
Without loss of generality, we assume $\inf_{t\in[0,\infty)} \lambda^\circ(t) = \lambda_0 > -\infty$ and $\lambda_* \ge \lambda_0$.
\end{assumption}

Assumptions \ref{asm:grad-Lip}--\ref{asm:coercive} are standard stability and regularity conditions for the population dynamics. Assumption~\ref{asm:grad-Lip} controls the local variability of the gradients. If the pointwise Hessian $\ddot{m}_{\theta} = \pdv*{m_\theta}{\theta}{\theta^\top}$ exists, then the envelope function can be taken as \[\dot{\psi}_r = \sup_{\theta\in\mathcal{B}(\theta_*,r)} \normop{\ddot{m}_{\theta}}\] whenever this envelope belongs to $L^2(P)$. 
Assumption~\ref{asm:Hess-Lip} ensures that the curvature of the population landscape varies smoothly, allowing us to compare the dynamics \eqref{eq:ode-func_int} with a linear system defined by the constant coefficient matrix $H(\theta_*)$. In combination with Assumption~\ref{asm:converge}, Weyl's inequality shows that Assumption~\ref{asm:coercive} follows from positive definiteness of $H(\theta_*)$. Thus, Assumption~\ref{asm:coercive} formalizes the requirement that, after the flow has entered a neighborhood of $\theta_*$, the population dynamics become contractive. The next lemma shows how this eventual contraction is used.

\begin{lemma}\label{lem:trans_mat_norm}
Under Assumption~\ref{asm:coercive}, with $A_* = \exp\{(\lambda_*-\lambda_0)t_*\}/\lambda_*$, the matrix $\Pi(t,s)$ in \eqref{eq:transition_matrix} satisfies 
\[ \normop{\Pi(t,s)} \le A_*\lambda_*\exp\{-\lambda_*(t-s)\} ,\quad \forall t \ge s \ge 0 .\]
\end{lemma}

We can now show the finite arc length of $(\Phi_t)_{t\in[0,\infty)}$. 
\begin{proposition}\label{prop:length_bound}
Under Assumptions \ref{asm:converge}--\ref{asm:coercive}, \[\int_0^\infty \norm{\norm{\pdv{\Phi_t}{t}}_{\ell^2}}_{L^2(P)} \dd{t} \le \bar{\ell},\] where 
\[ \bar{\ell} = (1 + \Lambda A_*) \{\lVert\dot{\psi}_{C_0}\rVert_{L^2(P)} + L(C_0) B_* /\lambda_*\} C_0 /\mu + B_* /\lambda_* ,\]
with
\[
\Lambda = \sup_{t\in[0,\infty)}\normop{H(\theta^\circ(t))}, \qquad A_* = \exp\{(\lambda_*-\lambda_0)t_*\}/\lambda_*, \qquad B_* = \norm{\norm{\psi_{\theta_*}}_{\ell^2}}_{L^2(P)}.
\]
\end{proposition}

\subsection{Main results}
\label{sec:main}
In the preceding subsection we showed that the leading term $n^{1/2}\Delta_n$ admits uniform Gaussian asymptotics. We now show that it also dominates the full gradient flow error.
Recall that 
\[ \norm{\theta}_{L^\infty([0,\infty))} = \sup_{t\in[0,\infty)} \norm{\theta(t)}_{\ell^2}\]
for any map $\theta : [0,\infty) \to \R^d$.
By Proposition~\ref{prop:length_bound}, the leading term \eqref{eq:lead} satisfies 
\[\begin{aligned}
\E\{\norm{\Delta_n}_{L^\infty([0,\infty))}\} 
&\le \int_0^\infty \E\bigg\{\norm{(P_n-P)\pdv{t}\Phi_t}_{\ell^2}\bigg\} \dd{t} \\
&\le \int_0^\infty n^{-1/2} \norm{\norm{\pdv{t}\Phi_t}_{\ell^2}}_{L^2(P)} \dd{t} 
\le n^{-1/2} \bar{\ell} .
\end{aligned}\]

\begin{theorem}[Time-uniform linearization of the empirical gradient flow]\label{thm:linearization}
Suppose that the matrix in \eqref{eq:transition_matrix} satisfies 
\begin{equation}\label{asm:transition_matrix}
\sup_{t\in[0,\infty)} \int_0^t \normop{\Pi(t,s)} \dd{s} \le A_* < \infty.
\end{equation}
Suppose further that the higher-order terms \eqref{eq:doubly-centered} and \eqref{eq:Bregman_div} satisfy 
\begin{equation}\label{asm:doubly-centered}
\norm{R_n(t)}_{\ell^2} \le \eta_n \left\{n^{-1/2} + \norm{\hat{\theta}(t)-\theta^\circ(t)}_{\ell^2}\right\} ,\quad \forall t\in[0,\infty),
\end{equation}
\begin{equation}\label{asm:Bregman_div}
\norm{D_\Psi\big(\hat{\theta}(t),\theta^\circ(t)\big)}_{\ell^2} \le \frac{L_n}{2} \norm{\hat{\theta}(t)-\theta^\circ(t)}_{\ell^2}^2 ,\quad \forall t\in[0,\infty),
\end{equation}
for some nonnegative random variables $\eta_n$ and $L_n$. 
Put $U_n = n^{1/2}\norm{\Delta_n}_{L^\infty([0,\infty))}$ and define the event
\[ \Omega_n = \{2 A_*\eta_n + 4 n^{-1/2} A_* L_n (U_n + A_* \eta_n) \le 1\} .\]
Then the following holds:
\begin{enumerate}
\item[(i)] On the event $\Omega_n$, 
\[ \norm{\hat{\theta}-\theta^\circ-\Delta_n}_{L^\infty([0,\infty))} \le n^{-1/2} A_*\eta_n (1+2U_n+2A_*\eta_n) + 2 n^{-1} A_* L_n (U_n+A_*\eta_n)^2 .\]
\item[(ii)] If $\eta_n = o_\Pr(1)$, $L_n = o_\Pr(n^{1/2})$, and $U_n = \mathcal{O}_\Pr(1)$, then 
$\Pr(\Omega_n) \to 1$ as $n\to\infty$, and 
\[ n^{1/2}(\hat{\theta}-\theta^\circ) = n^{1/2}\Delta_n + \varepsilon_n \quad\text{with}\quad \norm{\varepsilon_n}_{L^\infty([0,\infty))} = o_\Pr(1) .\]
\end{enumerate}
\end{theorem}

The proof of Theorem~\ref{thm:linearization} is based on a bootstrap argument
for differential equations \citep[Principle~1.23]{tao2006nonlinear}, in which
\textit{a priori} estimates are successively sharpened into improved ones. Conditions~\eqref{asm:transition_matrix},
\eqref{asm:doubly-centered}, and~\eqref{asm:Bregman_div} of the theorem are tailored to drive this argument and hold under primitive assumptions: First, under Assumption~\ref{asm:coercive}, Lemma~\ref{lem:trans_mat_norm} implies condition \eqref{asm:transition_matrix} with
\[A_* = \exp\{(\lambda_*-\lambda_0)t_*\}/\lambda_*.\]
Second, condition~\eqref{asm:doubly-centered} is analogous to standard sufficient conditions for the asymptotic normality of Z-estimators \citep[Theorem~3.3.1]{vVW2023weak}.
When the pointwise Hessian $\ddot{m}_\theta = \pdv*{m_\theta}{\theta}{\theta^\top}$ exists and is continuous in $\theta$, the remainder in
\eqref{eq:doubly-centered} admits the representation
\[\begin{aligned}
R_n(t)
&= (P_n-P)\{\psi_{\hat{\theta}(t)}-\psi_{\theta^\circ(t)}\} \\
&= \left\{\int_0^1 (P_n-P)\ddot{m}_{\alpha\hat{\theta}(t)+(1-\alpha)\theta^\circ(t)} \dd{\alpha}\right\}
\left\{\hat{\theta}(t)-\theta^\circ(t)\right\} .
\end{aligned}\]
Hence, if $\hat{\theta}$ and $\theta^\circ$ lie in
$\mathcal{B}(\theta_*,r)$ with probability tending to one for some constant $r>0$, we may take
\[\eta_n = \sup_{\theta\in\mathcal{B}(\theta_*,r)}\normop{(P_n-P)\ddot{m}_\theta}\]
in~\eqref{asm:doubly-centered}. Under sub-Gaussian and sub-exponential conditions on the gradient and Hessian fluctuations, together with an integrability condition on the Lipschitz constant of
$\theta\mapsto\ddot{m}_\theta(Z_1)$, \citet[Theorem~1]{mei2018landscape} gives
\begin{equation}\label{eq:Hess-rate}
\sup_{\theta\in\mathcal{B}(\theta_*,r)} \normop{(P_n-P)\ddot{m}_\theta} = \mathcal{O}_\Pr\left\{\Big(\frac{\log n}{n}\Big)^{1/2}\right\} .
\end{equation}
Third, condition~\eqref{asm:Bregman_div} follows with $L_n = L(r)$ from Assumption~\ref{asm:Hess-Lip}. Indeed, for any $\theta,\beta \in \mathcal{B}(\theta_*,r)$, by the mean value theorem, 
\[\begin{aligned}
\norm{D_\Psi(\theta,\beta)}_{\ell^2}
&= \norm{\int_0^1 \{H(\alpha\theta+(1-\alpha)\beta) - H(\beta)\}(\theta-\beta) \dd{\alpha}}_{\ell^2} \\
&\le \int_0^1 \normop{H(\alpha\theta+(1-\alpha)\beta) - H(\beta)}\,\norm{\theta-\beta}_{\ell^2} \dd{\alpha} \\
&\le \int_0^1 L(r)\,\alpha\,\norm{\theta-\beta}_{\ell^2}^2 \dd{\alpha}
= \frac{L(r)}{2}\norm{\theta-\beta}_{\ell^2}^2 .
\end{aligned}\]

Combining the Gaussian asymptotics for $n^{1/2}\Delta_n$ with the linearization bound in Theorem~\ref{thm:linearization} gives the following
uniform central limit theorem.

\begin{theorem}[Time-uniform central limit theorem for gradient flows]\label{thm:Gaussian}
Under the assumptions of Theorem~\ref{thm:linearization} and Proposition~\ref{prop:length_bound}, 
$n^{1/2}(\hat{\theta}-\theta^\circ)$ converges weakly in $L^\infty([0, \infty); \R^d)$ to a zero-mean Gaussian process with covariance function $G$ given by \eqref{eq:cov}.
\end{theorem}

Theorem~\ref{thm:Gaussian} gives a time-uniform asymptotic description of the statistical fluctuation of the empirical gradient flow around its population counterpart. 
The central observation is that the limiting process is governed not by the full ambient parameter space, but by the one-dimensional path traced out by the population gradient flow. This pathwise structure is what makes classical empirical process tools applicable to the dynamic optimization setting.

The theorem also yields inference at fixed and data-dependent stopping times. Let $\hat{t}$ be a possibly random time and let
$T\in[0,\infty)$ be deterministic. If $\hat{t}\to T$ in probability and the limiting Gaussian process $W$ is sample-continuous at $T$, then
\[n^{1/2}\{\hat{\theta}(\hat{t})-\theta^\circ(\hat{t})\} \xrightarrow{d} \N(0_d,G(T,T)).\]
If, in addition,
\[n^{1/2}\,\|\theta^\circ(\hat{t})-\theta^\circ(T)\|_{\ell^2} = o_{\Pr}(1),\]
then
\[n^{1/2}\{\hat{\theta}(\hat{t})-\theta^\circ(T)\} \xrightarrow{d} \N(0_d,G(T,T)).\]
Thus, the statistical variability of the stopped algorithmic output is governed by the covariance of the limiting Gaussian process at the effective stopping time.

The same argument extends to diverging stopping times by adjoining an endpoint at infinity. If $G(\infty,\infty):=\lim_{t\to\infty}G(t,t)$ exists, and if the limiting process is stochastically equicontinuous at infinity,
then the preceding conclusion extends to $\hat{t}\to\infty$ in probability, with $T=\infty$. This is the case relevant for stopping rules whose effective time horizon diverges with $n$.

Finally, if $\hat{G}$ is uniformly consistent for $G$ on $[0,\infty)\times[0,\infty)$, then $\hat{G}(\hat{t},\hat{t})$ consistently estimates the asymptotic covariance at the effective stopping time. Whenever centering by $\theta^\circ(T)$ is valid, this gives the large-sample approximation $\hat{\theta}(\hat{t}) \approx \N\big(\theta^\circ(T), n^{-1}\hat{G}(\hat{t},\hat{t})\big)$.

\section{Algorithmic covariance estimation and inference}
\label{sec:cov_est}
In this section we leverage the results from Section~\ref{sec:converge} to develop practical procedures for statistical inference along the gradient flow. The main task is to estimate, uniformly over time, the covariance function $G$ of the limiting Gaussian process. To this end, we introduce an auxiliary sensitivity dynamics that is solved jointly with the empirical gradient flow. The resulting covariance estimator is algorithm-aware: it evolves together with the optimization trajectory, rather than being computed only after optimization has terminated.

\subsection{Proposed methodology}
\label{sec:method}
Let $e\in\R^d$ be fixed. Theorem~\ref{thm:Gaussian} gives the asymptotic distribution of $n^{1/2}\,e^\top\{\hat{\theta}(t)-\theta^\circ(t)\}$ uniformly over $t\geq0$. Thus, for a deterministic or data-dependent stopping time $\hat{t}$, a natural Wald-type confidence interval for
$e^\top\theta^\circ(\hat{t})$ is
\begin{equation}\label{eq:CI-general}
e^\top\hat{\theta}(\hat{t}) \pm q_\alpha n^{-1/2} \{e^\top\hat{G}_n(\hat{t},\hat{t})e\}^{1/2} ,
\end{equation}
where $q_\alpha$ denotes the $(1-\alpha/2)$-quantile of the standard normal distribution and $\hat{G}_n$ is an estimator of $G$. 
Taking $e=e_k$, $k=1,\dots,d$, gives coordinatewise intervals
\begin{equation}\label{eq:CI}
e_k^\top\hat{\theta}(\hat{t}) \pm q_\alpha n^{-1/2}\{e_k^\top\hat{G}_n(\hat{t},\hat{t})e_k\}^{1/2} .
\end{equation}
The interval \eqref{eq:CI-general} is naturally constructed for $e^\top\theta^\circ(\hat{t})$. It is also valid for $e^\top\theta_*$ when the population early-stopping bias is negligible at the $n^{-1/2}$-scale, namely when
\[n^{1/2}\,|e^\top\{\theta^\circ(\hat{t})-\theta_*\}| = o_{\Pr}(1).\]
Thus, inference for the limiting target $\theta_*$ requires two ingredients: the time-uniform approximation of the stochastic fluctuation $\hat{\theta}-\theta^\circ$, provided by Theorem~\ref{thm:Gaussian}, and sufficiently small deterministic population bias along the population flow, implied by Assumption~\ref{asm:converge}.

It remains to construct an estimator $\hat{G}_n$ that is uniformly consistent for the covariance function $G$. Recall from \eqref{eq:cov} that
\[G(t_1,t_2) = \Cov_P(\Phi_{t_1},\Phi_{t_2}) = P\Phi_{t_1}\Phi_{t_2}^\top - P\Phi_{t_1} P\Phi_{t_2}^\top.\]
We define the oracle empirical covariance 
\[ G_n(t_1,t_2) = \Cov_{P_n}(\Phi_{t_1},\Phi_{t_2}) = P_n\Phi_{t_1}\Phi_{t_2}^\top - P_n\Phi_{t_1} P_n\Phi_{t_2}^\top ,\quad t_1,t_2\in[0,\infty) .\]
This oracle estimator is infeasible because the sensitivity functions $\Phi_t$ are unknown. Even in ordinary least squares they depend on population quantities such as $\Sigma$ and $\xi$; in more general models, they may not even admit a closed-form expression.
However, $(\Phi_t)_{t\in[0,\infty)}$ is characterized as the solution of the differnetial equation \eqref{eq:ode-func_int}. This suggests estimating $\Phi_t$ through the corresponding empirical dynamics. Define $\hat{\Phi}_t$ as the solution of the differential equation 
\begin{equation}\label{eq:ode-empirical}
\pdv{t}\hat{\Phi}_t = - \hat{H}(\hat{\theta}(t)) \hat{\Phi}_t + \psi_{\hat{\theta}(t)} ,\quad \hat{\Phi}_0=0_d ,
\end{equation}
where $\hat{H}$ is a consistent estimator for $H$, e.g., $\hat{H}(\theta) = P_n \ddot{m}_\theta$ in view of \eqref{eq:Hess-rate}. 
This mirrors the dynamics of $\Phi_t$, with population quantities replaced by their empirical counterparts, and can be computed alongside the gradient flow $\hat{\theta}$ with moderate additional cost.
Evaluating the $\hat{\Phi}_t$'s at the observed data points, we then propose an estimator $\hat{G}_n$ defined by 
\[ \hat{G}_n(t_1,t_2) = P_n\hat{\Phi}_{t_1}\hat{\Phi}_{t_2}^\top - P_n\hat{\Phi}_{t_1} P_n\hat{\Phi}_{t_2}^\top ,\quad t_1,t_2\in[0,\infty) ,\]
which inherits the same plug-in structure as the oracle estimator $G_n$. 
Notably, $\hat{G}_n$ can be updated recursively along the trajectory, allowing simultaneous optimization and covariance estimation while avoiding matrix inversion, resampling, and sample splitting.

\subsection{Theoretical guarantee}
\label{sec:theory}
Next, we justify the proposed inference procedure by establishing that $\hat{G}_n$ is uniformly consistent for $G$ over $[0,\infty)\times[0,\infty)$. 

\begin{theorem}\label{thm:cov}
Suppose that the assumptions of Theorem~\ref{thm:Gaussian} hold. Moreover, let Assumption~\ref{asm:grad-Lip} be strengthened such that $\norm{\psi_\theta}_{\ell^2}$ and $\dot{\psi}_r$ belong to $L^4(P)$. 
Consider the plug-in estimator for the population Hessian $\hat{H}(\theta) = P_n \ddot{m}_\theta$ and suppose that \eqref{eq:Hess-rate} holds.
Then the proposed covariance estimator $\hat{G}_n$ satisfies 
\[ \sup_{t_1,t_2\in[0,\infty)} \normop{\hat{G}_n(t_1,t_2) - G(t_1,t_2)} = \mathcal{O}_\Pr\left\{\Big(\frac{\log n}{n}\Big)^{1/2}\right\} .\]
\end{theorem}

The logarithmic factor in Theorem~\ref{thm:cov} originates entirely from \eqref{eq:Hess-rate}. It can be removed if the population Hessian $H$ admits an estimator that is uniformly consistent at root-$n$ rate. 
The theorem shows that the infinite time horizon does not introduce an additional statistical price beyond this mild logarithmic term.
This reflects the intrinsic low-dimensional structure of the gradient flow trajectory, which enables uniform control despite the infinite time horizon. 

In the remainder of this section, we briefly describe our proof strategy: The estimation error can be decomposed into 
\[ \hat{G}_n-G = (\hat{G}_n-G_n) + (G_n-G) ,\]
where the first term captures the error induced by approximating $\Phi_t$ with $\hat{\Phi}_t$, while the second term corresponds to the usual empirical process fluctuation. We bound these two components separately.

We first control the estimation error $\hat{\Phi}_t-\Phi_t$. 
The key observation is that $\Phi_t$ admits an integral representation. 
For each $s\in[0,\infty)$, define $\hat{\phi}_{t,s}$, $t\ge s$, by the initial value problem 
\[ \pdv{t}\hat{\phi}_{t,s} = - \hat{H}(\hat{\theta}(t)) \hat{\phi}_{t,s} ,\quad \hat{\phi}_{s,s}=\psi_{\hat{\theta}(s)} .\]
\begin{lemma}\label{lem:ode}
The solution $\hat{\Phi}_t$ of \eqref{eq:ode-empirical} can be represented as 
\[ \hat{\Phi}_t(\cdot) = \int_{s=0}^{t} \hat{\phi}_{t,s}(\cdot) \dd{s} .\]
\end{lemma}

Lemma~\ref{lem:ode} shows that $\hat{\Phi}_t$ can be viewed as an accumulation of gradients propagated along the empirical flow, in direct analogy with the population representation \eqref{eq:func_int} of $\Phi_t$. 
This formulation allows us to compare $\hat{\Phi}_t$ and $\Phi_t$ through their respective evolution operators, and to take advantage of the stability properties of the underlying homogeneous differential equations.

\begin{lemma}\label{lem:est1_cov}
Under Assumption~\ref{asm:coercive}, with \[ A_* = \exp\{(\lambda_*-\lambda_0)t_*\}/\lambda_* ,\quad B_n = \sup_{t\in[0,\infty)} \norm{\norm{\psi_{\theta^\circ(t)}}_{\ell^2}}_{L^2(P_n)} ,\] \[ \kappa_n = \sup_{t\in[0,\infty)} \normop{\hat{H}(\hat{\theta}(t))-H(\theta^\circ(t))} ,\quad \sigma_n = \sup_{t\in[0,\infty)} \norm{\norm{\psi_{\hat{\theta}(t)}-\psi_{\theta^\circ(t)}}_{\ell^2}}_{L^2(P_n)} ,\]
where $\norm{f}_{L^2(P_n)} = (P_n f^2)^{1/2}$, 
it holds on the event $\{A_*\kappa_n<1\}$ that 
\[ \sup_{t_1,t_2\in[0,\infty)} \normop{\hat{G}_n(t_1,t_2) - G_n(t_1,t_2)} \le 4 A_*^2 B_n c_n + 2 A_*^2 c_n^2 ,\]
where $c_n = A_*B_n\kappa_n/(1-A_*\kappa_n)^2 + \sigma_n/(1-A_*\kappa_n)$. 
\end{lemma}

Lemma~\ref{lem:est1_cov} shows that the discrepancy $\hat{G}_n-G_n$ is governed by two quantities, $\kappa_n$ and $\sigma_n$ that measure the uniform error in approximating the population Hessian and the individual gradients, respectively.
The factor $B_n$ provides a uniform bound on the empirical norm of the gradients. If $B_n$ is bounded in probability and both $\kappa_n$ and $\sigma_n$ vanish in probability, then $\hat{G}_n-G_n$ is negligible.

It remains to control $G_n-G$, the empirical fluctuation of the oracle estimator.
\begin{lemma}\label{lem:est2_cov}
If \eqref{asm:transition_matrix} holds, then 
\[ \sup_{t_1,t_2\in[0,\infty)} \normop{G_n(t_1,t_2) - G(t_1,t_2)} \le A_*^2 (\gamma_n + 2 B^\circ \zeta_n + \zeta_n^2) ,\]
where $B^\circ = \sup_{t\in[0,\infty)} \norm{P\psi_{\theta^\circ(t)}}_{\ell^2}$, and 
\[ \zeta_n = \sup_{t\in[0,\infty)} \norm{(P_n-P)\psi_{\theta^\circ(t)}}_{\ell^2} ,\quad \gamma_n = \sup_{s,t\in[0,\infty)} \normop{(P_n-P)\psi_{\theta^\circ(s)}\psi_{\theta^\circ(t)}^\top} .\]
\end{lemma}

Combining Lemmas~\ref{lem:est1_cov} and~\ref{lem:est2_cov} with the uniform Hessian bound \eqref{eq:Hess-rate} yields Theorem~\ref{thm:cov}. Consequently, the covariance estimator $\hat{G}_n$ is uniformly consistent and can be used in the Wald interval \eqref{eq:CI-general} for inference along the gradient flow, including at data-dependent and diverging stopping times for which the population bias is controlled.

\section{Numerical results}
\label{sec:numerical}
In this section, we assess the practical performance of the proposed approach through numerical experiments.
First we describe the computational procedures used to solve the gradient flow equation and the auxiliary dynamics that underpin our covariance estimation method. 
Then we present simulation results that illustrate the finite-sample accuracy of the Gaussian approximation and the coverage of the resulting confidence intervals.

\subsection{Implementation}
\label{sec:algo}
Since both \eqref{eq:gradflow_emp} and \eqref{eq:ode-empirical}, the gradient flow and the auxiliary dynamics, evolve jointly over time, it is natural to discretize them simultaneously. 
We consider two widely used schemes, the explicit Euler method and a fourth-order Runge--Kutta method \citep{atkinson2009numerical}. The Euler scheme serves as a simple and computationally efficient baseline, while the Runge--Kutta method improves accuracy by incorporating intermediate evaluations of the vector field. 
Notably, the explicit Euler discretization of \eqref{eq:gradflow_emp} coincides with the classical gradient descent algorithm with a fixed step size, thereby providing a direct connection between our continuous-time formulation and practical optimization procedures. 
We present these solvers in Algorithms \ref{alg:Euler} and \ref{alg:RK4}, highlighting how the primary trajectory $\hat{\theta}(t)$ and the auxiliary process $\hat{\Phi}_t$ can be updated in tandem at each iteration. To save space, Algorithm~\ref{alg:RK4} is deferred to Appendix~\ref{appn:numeric}.

\begin{algorithm}[!ht]
\caption{Euler's method to solve \eqref{eq:gradflow_emp} and \eqref{eq:ode-empirical}}
\label{alg:Euler}
\KwIn{Data $Z_1,\dots,Z_n$, initial value $\theta_0$, step size $\delta$.}
Initialization: set $J\gets 0$, $\hat{\theta}(0)\gets \theta_0$, and $\hat{\Phi}_0(Z_i) \gets 0_d$ for $i=1,\dots,n$\;
\Repeat{a stopping criterion is met, e.g., $J$ exceeds a maximum number of iterations or the Euclidean norm of $n^{-1}\sum_{i=1}^{n} g_i$ is less than a tolerance.}{
Compute the gradients $g_i \gets \psi_\theta(Z_i)$, $i=1,\dots,n$, at $\theta = \hat{\theta}(J\delta)$\;
Construct an estimate $\hat{H}$ for the population Hessian $H(\theta)$ at $\theta = \hat{\theta}(J\delta)$\;
Update 
\begin{align*}
& \hat{\theta}((J+1)\delta) \gets \hat{\theta}(J\delta) - \delta \cdot n^{-1}\sum_{i=1}^{n} g_i ;\\
& \hat{\Phi}_{(J+1)\delta}(Z_i) \gets \hat{\Phi}_{J\delta}(Z_i) + \delta \cdot \{g_i - \hat{H} \hat{\Phi}_{J\delta}(Z_i)\} ,\quad i=1,\dots,n ;\\
& J \gets J+1 ;
\end{align*}
}
\KwOut{Trajectories of estimators $\hat{\theta}(t)$ and $\hat{\Phi}_t(Z_i)$, $i=1,\dots,n$, with $t=j\delta$, $j=1,\dots,J$.}
\end{algorithm}

\subsection{Simulation}
\label{sec:simu}
We apply the method to three standard models from Example~\ref{ex:glr}, together with two extensions of the linear model. In each setting, $\theta_*$ denotes the stationary point of the population gradient flow~\eqref{eq:gradflow_pop}.
Let $X\in\R^{d_0}$ satisfy $X\sim \N(0_{d_0},\Sigma)$ where $e_j^\top\Sigma e_k = 1.09\,\1\{j=k\} + 0.6\,\1\{j+k=d_0+1\}$, and fix $\beta\in\R^{d_0}$ with $e_k^\top\beta = 2+\1\{k\in 2\mathbb{N}\}$. 

\begin{itemize}
\item Linear regression (loss $l_{\mathrm{lin}}$): 
$Y\mid X\sim \N(X^\top\beta,\sigma^2)$, $\sigma=0.1$, $\theta_*=\beta$.

\item Logistic regression (loss $l_{\mathrm{logit}}$): 
$Y\mid X\sim \mathrm{Bernoulli}\!\left(\{1+\exp\,\!(-X^\top\beta)\}^{-1}\right)$, $\theta_*=\beta$.

\item Phase retrieval (loss $l_{\mathrm{pr}}$): 
$Y\mid X\sim \N\big((X^\top\beta)^2,\sigma^2\big)$, $\sigma=0.5$, $\theta_*=\beta$, selected over $-\beta$ by initializing at a point with positive coordinates.

\item Quantile regression (loss $l_{\mathrm{qt}}$, level $\tau=0.78$): 
linear-model data with $Y=X^\top\beta+\varepsilon$, $\varepsilon\sim \N(0,\sigma^2)$, $\sigma=0.1$, augmented covariates $\widetilde X=(1,X^\top)^\top$, and $\theta_*=(\sigma z_\tau,\beta^\top)^\top$, where $z_\tau$ is the standard normal $\tau$-quantile.

\item Ridge regression (loss $l_{\mathrm{lin}}$ with penalty $p(\theta)=(\lambda/2)\norm{\theta}_{\ell^2}^2$, $\lambda=0.123$): 
linear-model data with $Y\mid X\sim \N(X^\top\beta,\sigma^2)$, $\sigma=0.1$, and $\theta_*=(\Sigma+\lambda I_{d_0})^{-1}\Sigma\beta$, exhibiting the shrinkage bias induced by regularization.
\end{itemize}

We set the sample size and dimension to be $n = 1000$ and $d_0 = 5$, respectively, and implement Algorithms \ref{alg:Euler} and \ref{alg:RK4} with step size $\delta = 10^{-3}$. 
Given a twice differentiable loss, we estimate the population Hessian by its empirical counterpart $\hat{H}(\theta) = P_n \ddot{m}_\theta$. In quantile regression, we use 
\[ \hat{H}(\theta) = \{nh(\theta)\}^{-1}\sum_{i=1}^{n}\varphi\big(r_i(\theta)/h(\theta)\big) \cdot n^{-1}\sum_{i=1}^{n}(1,X_i^\top)^\top(1,X_i^\top) ,\]
where $\varphi$ is the probability density function of the standard normal distribution, $r_i(\theta)$ is the residual when predicting the $i$th subject with parameter $\theta$, and $h(\theta) = \hat{\operatorname{sd}}(\theta)n^{-1/5}$ where $\hat{\operatorname{sd}}(\theta)$ is the sample standard deviation of $r_i(\theta)$'s. 
We initialize the algorithms at $\theta_0 = \beta_0$ with $e_k^\top\beta_0 = 2.5$ for all $k = 1 , \dots, d_0$, except in quantile regression where $\theta_0 = (0.1, \beta_0^\top)^\top$. 
We cap the number of iterations at $2\times 10^6$ and stop when the gradient norm falls below $10^{-5}$. The resulting horizon is sufficiently long that the population flow is effectively stationary at the terminal time $\hat{t}$.
Inference for each coordinate $e_k^\top\theta_*$ is conducted via the confidence interval \eqref{eq:CI}. Equivalently, we consider the standardized z-scores 
\begin{equation}\label{eq:z-score}
\hat{z}_k = e_k^\top\{\hat{\theta}(\hat{t})-\theta_*\} / \{n^{-1} e_k^\top\hat{G}_n(\hat{t},\hat{t})e_k\}^{1/2} ,
\end{equation}
which our theory predicts to be approximately standard normal.

To assess the finite-sample performance of our proposed method, we conduct $1000$ Monte Carlo replications. 
The distributional results are reported in Figures \ref{fig:Euler} and \ref{fig:RK4} in Appendix~\ref{appn:numeric}, with Figure~\ref{fig:demo} showing a representative excerpt. 
The density estimates and the Q-Q plots of the standardized scores in~\eqref{eq:z-score} closely match the standard normal approximation.
Table~\ref{tab:coverage} reports the average coverage rates of the confidence intervals in \eqref{eq:CI} at the nominal $90\%$ and $95\%$ levels.
The results demonstrate that the proposed inference procedures achieve accurate coverage across all models considered, and remain stable under different numerical solvers. This supports both the theoretical Gaussian asymptotics and the robustness of the proposed implementation. 

\begin{figure}[!ht]
    \centering
\includegraphics[width=\linewidth]{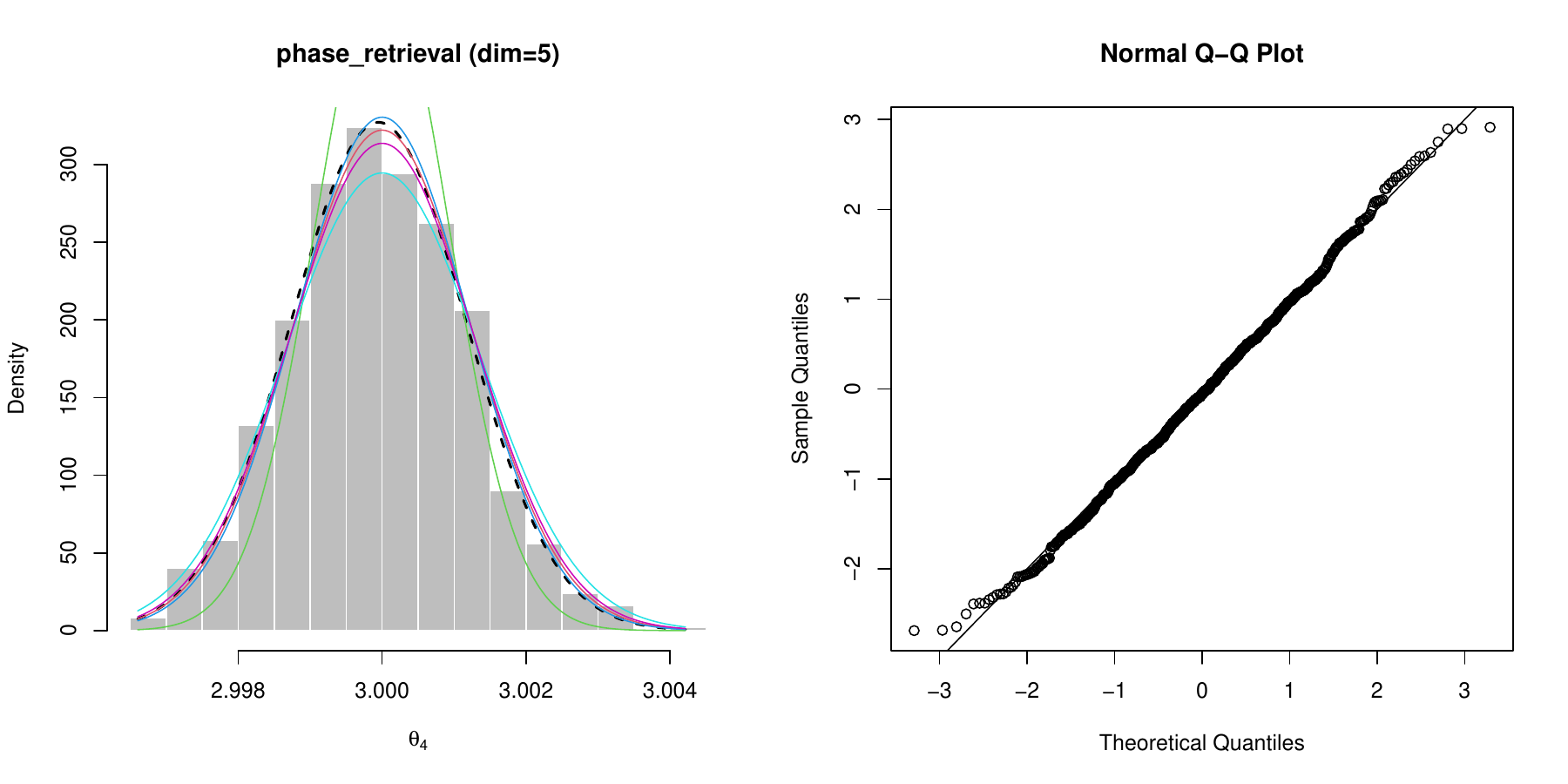}
\caption{Fourth coordinate of stationary points of gradient flows solved by Algorithm~\ref{alg:Euler} in phase retrieval. The dashed black line is the normal density plot with mean and variance based on Monte Carlo replications, the solid colored lines are the normal density plots with mean to be the population truth and variance to be the estimates in several Monte Carlo runs, and the Q-Q plot consists of the z-scores \eqref{eq:z-score} arising from Monte Carlo replications versus theoretical normal quantiles.}
\label{fig:demo}
\end{figure}
\begin{table}[!ht]
\caption{Average coverage rates (\%) under either Algorithm~\ref{alg:Euler} or Algorithm~\ref{alg:RK4}.}
\label{tab:coverage}
    \centering
\begin{tabular}{l|rrrrrr}
\hline
(Nominal 90\%) & $e_1^\top\theta_*$ & $e_2^\top\theta_*$ & $e_3^\top\theta_*$ & $e_4^\top\theta_*$ & $e_5^\top\theta_*$ & $\sigma z_\tau$ \\
Linear regression & 89.7 & 89.1 & 86.8 & 89.7 & 90.6 & NA\\
Logistic regression & 90.2 & 91.1 & 89.5 & 88.7 & 89.4 & NA\\
Phase retrieval & 91.5 & 89.3 & 89.7 & 89.1 & 90.6 & NA\\
Quantile regression & 89.0 & 90.1 & 88.6 & 91.1 & 89.2 & 88.6 \\
Ridge regression & 89.8 & 89.1 & 88.2 & 90.7 & 91.1 & NA\\
\hline(Nominal 95\%) & $e_1^\top\theta_*$ & $e_2^\top\theta_*$ & $e_3^\top\theta_*$ & $e_4^\top\theta_*$ & $e_5^\top\theta_*$ & $\sigma z_\tau$ \\
Linear regression & 94.8 & 95.1 & 93.4 & 94.6 & 94.8 & NA\\
Logistic regression & 95.5 & 96.4 & 95.5 & 94.4 & 94.5 & NA\\
Phase retrieval & 96.0 & 94.6 & 95.0 & 94.1 & 95.7 & NA\\
Quantile regression & 94.6 & 95.3 & 94.0 & 95.1 & 93.7 & 94.1 \\
Ridge regression & 94.8 & 94.7 & 93.7 & 95.5 & 95.2 & NA\\
\hline
\end{tabular}
\end{table}

We emphasize that our theoretical results are derived in a large-sample regime, where the sample size $n$ is sufficiently large relative to the ambient dimension. In additional numerical experiments with fixed $n$ and increasing dimension $d_0$, we observe that the Gaussian approximation becomes less reliable. 
This deterioration is most pronounced for logistic regression and phase retrieval, where the gradient descent solutions exhibit noticeable instability. 
Such a phenomenon can be attributed to the worsening of curvature properties in high dimensions with significant noise. The empirical loss landscape may become nearly flat or ill-conditioned in certain directions, leading to amplified stochastic fluctuations and weaker concentration of the gradient. 
This requires careful additional treatment, such as bias reduction \citep{stolte2024comprehensive} and spectral initialization \citep{peng2024noisy}. 
As a consequence, the linearization underlying our uniform central limit theorem provides a less accurate approximation, and the resulting inference procedures may not perform well.

\section{Discussion}
\label{sec:discuss}

Two aspects of our results merit further discussion. First, the analysis treats optimization algorithms as stochastic processes indexed by time, rather than merely as procedures that return a terminal estimator. 
This perspective makes the statistical content of the trajectory explicit and allows its pathwise structure to be used for uniform convergence and inference. In particular, the low-complexity geometry of gradient flow paths plays a crucial role in enabling the application of empirical process techniques over infinite time intervals.

Second, the proposed covariance estimator integrates inference directly into the algorithmic dynamics. This stands in contrast to traditional approaches that treat optimization and inference as separate stages. By coupling these two components, we obtain procedures that are both computationally efficient and statistically robust, especially in settings where the stopping rule is adaptive or implicitly defined.

There are several potential directions for future research. 
First, our analysis applies to fixed-dimensional regimes and relies on structural conditions such as local strong convexity and Lipschitz continuity of the gradient and Hessian. These assumptions ensure stability of the gradient flow and make the associated empirical-process analysis tractable. In high-dimensional regimes, where the ambient dimension $d$ may grow with or exceed the sample size $n$, these conditions typically fail globally and must be replaced by weaker, structure-adapted assumptions. A natural extension is to consider parameters with low-complexity structure such as sparsity, low rank, or manifold constraints. Then restricted strong convexity or localized curvature conditions can be imposed along the trajectory to ensure that the gradient flow remains confined to a region of controlled effective dimension. In this context, regularization plays a significant role. Penalties such as $\ell^1$- or nuclear norms introduce non-smooth optimization landscapes, so the gradient flow is more appropriately formulated in terms of differential inclusion. Although existence and stability can still be established under suitable convexity assumptions, the resulting statistical behavior may be substantially altered by active-set dynamics and model selection effects. Extending our time-uniform inference framework to such settings would therefore require combining dynamical analysis with high-dimensional tools, while simultaneously tracking both the evolution of the parameters and potential changes in the underlying model structure.

Second, an important direction concerns discretized stochastic algorithms encountered in practice, including stochastic gradient descent and its variants. In these settings, the dynamics is influenced by both the objective function and algorithmic noise, whose interaction with statistical variability can be subtle, particularly over long time horizons. While our analysis is conducted in continuous time, practical implementations rely on discrete updates with a finite step size, introducing discretization error that may accumulate and remain comparable to stochastic fluctuations. As a result, the gap between the discrete trajectory and its gradient flow limit can affect both the limiting distribution and the validity of inference procedures. Addressing this challenge requires a refined analysis that simultaneously controls discretization bias and stochastic error, potentially via coupling techniques or higher-order approximations of the dynamics. 

Finally, it is of interest to investigate whether the low-complexity structure of gradient flow trajectories, which underpins our uniform convergence results, persists in broader classes of algorithms. For example, proximal methods, mirror descent, and second-order schemes may exhibit different geometric properties that influence both their statistical behavior and the feasibility of time-uniform inference. Understanding these differences will help delineate the scope of algorithmic inference and identify general principles that govern the interaction between optimization dynamics and statistical uncertainty.



\appendix

\forcsvlist{\renewAformat}{assumption,lemma,theorem,corollary,equation}

\section{Donsker property of one-parameter functions}
\label{appn:Donsker}
We establish a uniform central limit theorem for empirical processes indexed by one-parameter functions, stated in Lemma~\ref{lem:uclt} below.
The newly introduced criterion of arc length extends the regularly used function classes \citep[Chapter~2.7]{vVW2023weak}.

\begin{lemma}\label{lem:uclt}
Let $\varphi_t$, $t\in \mathcal{T}$, be real-valued functions with index set $\mathcal{T}\subset\R$ being an interval. 
Assume that $\varphi_{t_0} \in L^2(P)$ for some $t_0 \in \mathcal{T}$, and that $\pdv*{\varphi_t(\cdot)}{t}$ exists and satisfies the finite arc length condition:
\[ \int_\mathcal{T} \norm{\pdv{t}\varphi_t}_{L^2(P)} \dd{t} < \infty ,\]
where $\norm{f}_{L^2(P)} = (P f^2)^{1/2}$. 
Then $(\varphi_t)_{t\in \mathcal{T}}$ is $P$-Donsker in the sense that $n^{1/2}(P_n-P)\varphi_t$, $t\in \mathcal{T}$, converges in distribution to a zero-mean Gaussian process uniformly on $\mathcal{T}$.
\end{lemma}

If the one-dimensional index set is a finite interval, then the arc length is easily bounded for continuously differentiable trajectories. 
To address the more interesting case of an infinite interval, we provide a method to control the arc length of a differential dynamical system. Assuming that the forcing vector field converges fast enough to a limiting autonomous field with a strictly stable equilibrium, we can make the velocity integrable by revealing a scalar dissipative inequality of the autonomous system. Such dissipation is built on the evolution of energy, which is typically described by a Lyapunov function \citep[Section~3.4]{barreira2012ordinary}. 
The arc length bound is presented in the following lemma. 

\begin{lemma}\label{lem:length-ode}
Let $\varphi : \mathcal{T} = [t_0,\infty) \to E$ be a mapping with bounded range $\varphi(\mathcal{T})$, where $E$ is a Hilbert space endowed with inner product $\inprod{\cdot}{\cdot}$ and norm $\norm{\cdot}$. Suppose that 
\[ \dv{\varphi(t)}{t} = F(\varphi(t),t) \]
for a mapping $F : E\times\mathcal{T} \to E$ with separable, integrable non-autonomous defect, in the sense that there exists a mapping $F_* : E \to E$, a non-decreasing function $a : [0,\infty) \to [0,\infty)$ and an integrable function $r : \mathcal{T} \to [0,\infty)$ such that 
\[ \norm{F(u,t)-F_*(u)} \le a(\norm{u}) r(t) .\]
Assume that $F_*$ has sub-linear growth around a stationary point $u_* \in E$, i.e., there exists a constant $\Lambda > 0$ such that 
\[ \norm{F_*(u)} \le \Lambda \norm{u-u_*} .\]
Assume further a differentiable Lyapunov function $V : E \to [0,\infty)$ such that for all $u \in \varphi(\mathcal{T})$, 
\begin{enumerate}
\item $\norm{u-u_*}^2/C_1 \le V(u) \le \norm{u-u_*}^2/C_2$ for some constants $C_1,C_2>0$;
\item $\inprod{\nabla V(u)}{F_*(u)} \le -\lambda \norm{u-u_*}^2$ for a constant $\lambda>0$;
\item $\norm{\nabla V(u)} \le L \norm{u-u_*}$ for a constant $L > 0$.
\end{enumerate}
Then the arc length of $\varphi$ is bounded by 
\[ \int_{\mathcal{T}} \norm{\dv{\varphi(t)}{t}} \dd{t} \le \frac{2\Lambda C_1^{1/2}}{\lambda C_2^{3/2}} \norm{\varphi(t_0)-u_*} + \left(1 + \frac{\Lambda L C_1}{\lambda C_2}\right) a\Big(\sup_{t\in\mathcal{T}}\norm{\varphi(t)}\Big) \int_{\mathcal{T}} r(t) \dd{t} .\]
\end{lemma}

The Lyapunov function is often implied by the physical or energetic structure of the dynamical system. In practice, one may simply try the ansatz that $V(u)$ be a quadratic form of $u-u_*$, especially for approximately linear dynamics. 
Note that an analogue of Proposition~\ref{prop:length_bound} can be deduced from Lemma~\ref{lem:length-ode}.

\section{Additional numerical results}
\label{appn:numeric}
This section consists of Algorithm~\ref{alg:RK4} and Figures \ref{fig:Euler} and \ref{fig:RK4} to complement Section~\ref{sec:numerical}.

\begin{algorithm}[!ht]
\caption{Runge--Kutta method of order $4$ to solve \eqref{eq:gradflow_emp} and \eqref{eq:ode-empirical}}
\label{alg:RK4}
\KwIn{Data $Z_1,\dots,Z_n$, initial value $\theta_0$, step size $\delta$.}
Initialization: set $J\gets 0$, $\hat{\theta}(0)\gets \theta_0$, and $\hat{\Phi}_0(Z_i) \gets 0_d$ for $i=1,\dots,n$\;
\Repeat{a stopping criterion is met.}{
Compute the gradients $g_i^{(1)} \gets \psi_\theta(Z_i)$, $i=1,\dots,n$, at $\theta = \hat{\theta}(J\delta)$\;
Construct an estimate $\hat{H}^{(1)}$ for the population Hessian $H(\theta)$ at $\theta = \hat{\theta}(J\delta)$\;
Update 
\begin{align*}
& \hat{\theta}^{(1)} \gets \hat{\theta}(J\delta) - (\delta/2) \cdot n^{-1}\sum_{i=1}^{n} g_i^{(1)} ;\\
& \hat{\Phi}^{(1)}(Z_i) \gets \hat{\Phi}_{J\delta}(Z_i) + (\delta/2) \cdot \{g_i^{(1)} - \hat{H}^{(1)} \hat{\Phi}_{J\delta}(Z_i)\} ,\quad i=1,\dots,n ;
\end{align*}
Compute the gradients $g_i^{(2)} \gets \psi_\theta(Z_i)$, $i=1,\dots,n$, at $\theta = \hat{\theta}^{(1)}$\;
Construct an estimate $\hat{H}^{(2)}$ for the population Hessian $H(\theta)$ at $\theta = \hat{\theta}^{(1)}$\;
Update 
\begin{align*}
& \hat{\theta}^{(2)} \gets \hat{\theta}(J\delta) - (\delta/2) \cdot n^{-1}\sum_{i=1}^{n} g_i^{(2)} ;\\
& \hat{\Phi}^{(2)}(Z_i) \gets \hat{\Phi}_{J\delta}(Z_i) + (\delta/2) \cdot \{g_i^{(2)} - \hat{H}^{(2)} \hat{\Phi}^{(1)}(Z_i)\} ,\quad i=1,\dots,n ;
\end{align*}
Compute the gradients $g_i^{(3)} \gets \psi_\theta(Z_i)$, $i=1,\dots,n$, at $\theta = \hat{\theta}^{(2)}$\;
Construct an estimate $\hat{H}^{(3)}$ for the population Hessian $H(\theta)$ at $\theta = \hat{\theta}^{(2)}$\;
Update 
\begin{align*}
& \hat{\theta}^{(3)} \gets \hat{\theta}(J\delta) - \delta \cdot n^{-1}\sum_{i=1}^{n} g_i^{(3)} ;\\
& \hat{\Phi}^{(3)}(Z_i) \gets \hat{\Phi}_{J\delta}(Z_i) + \delta \cdot \{g_i^{(3)} - \hat{H}^{(3)} \hat{\Phi}^{(2)}(Z_i)\} ,\quad i=1,\dots,n ;
\end{align*}
Compute the gradients $g_i^{(4)} \gets \psi_\theta(Z_i)$, $i=1,\dots,n$, at $\theta = \hat{\theta}^{(3)}$\;
Construct an estimate $\hat{H}^{(4)}$ for the population Hessian $H(\theta)$ at $\theta = \hat{\theta}^{(3)}$\;
Update 
\begin{align*}
& g_i \gets \{g_i^{(1)}+2g_i^{(2)}+2g_i^{(3)}+g_i^{(4)}\}/6 ,\quad i=1,\dots,n ;\\
& h_i \gets \{\hat{H}^{(1)} \hat{\Phi}_{J\delta}(Z_i) + 2 \hat{H}^{(2)} \hat{\Phi}^{(1)}(Z_i) + 2 \hat{H}^{(3)} \hat{\Phi}^{(2)}(Z_i) + \hat{H}^{(4)} \hat{\Phi}^{(3)}(Z_i)\} / 6 ,\quad i=1,\dots,n ;\\
& \hat{\theta}((J+1)\delta) \gets \hat{\theta}(J\delta) - \delta \cdot n^{-1}\sum_{i=1}^{n} g_i ;\\
& \hat{\Phi}_{(J+1)\delta}(Z_i) \gets \hat{\Phi}_{J\delta}(Z_i) + \delta \cdot (g_i - h_i) ,\quad i=1,\dots,n ;\\
& J \gets J+1 ;
\end{align*}
}
\KwOut{Trajectories of estimators $\hat{\theta}(t)$ and $\hat{\Phi}_t(Z_i)$, $i=1,\dots,n$, with $t=j\delta$, $j=1,\dots,J$.}
\end{algorithm}

\begin{figure}[!ht]
    \centering
\includegraphics[width=0.69\linewidth]{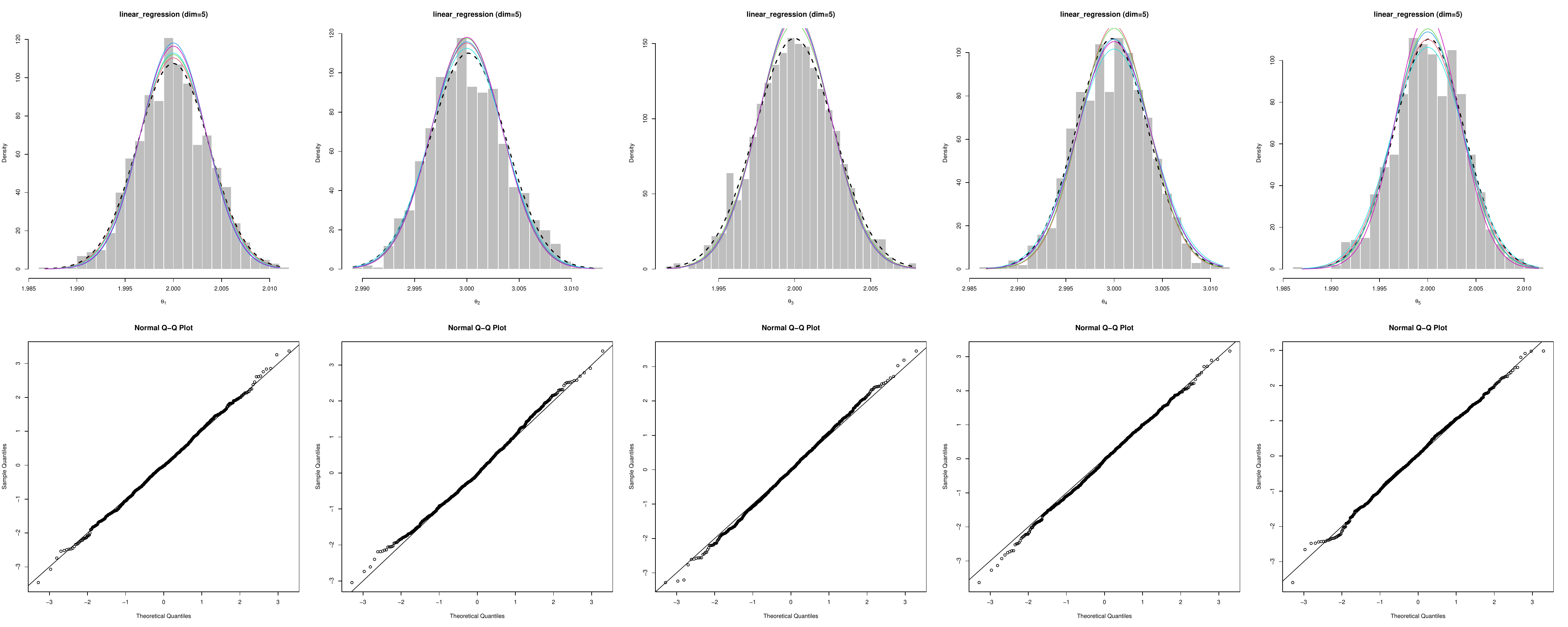}
\includegraphics[width=0.69\linewidth]{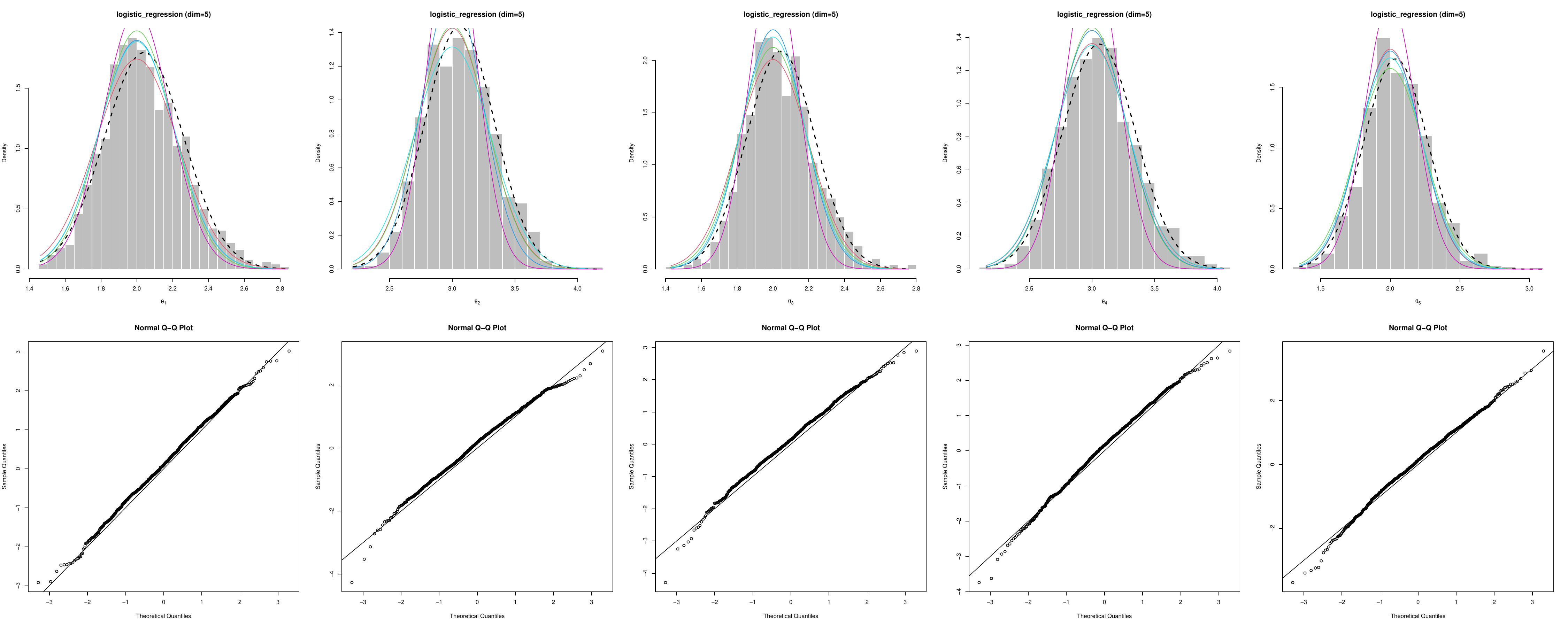}
\includegraphics[width=0.69\linewidth]{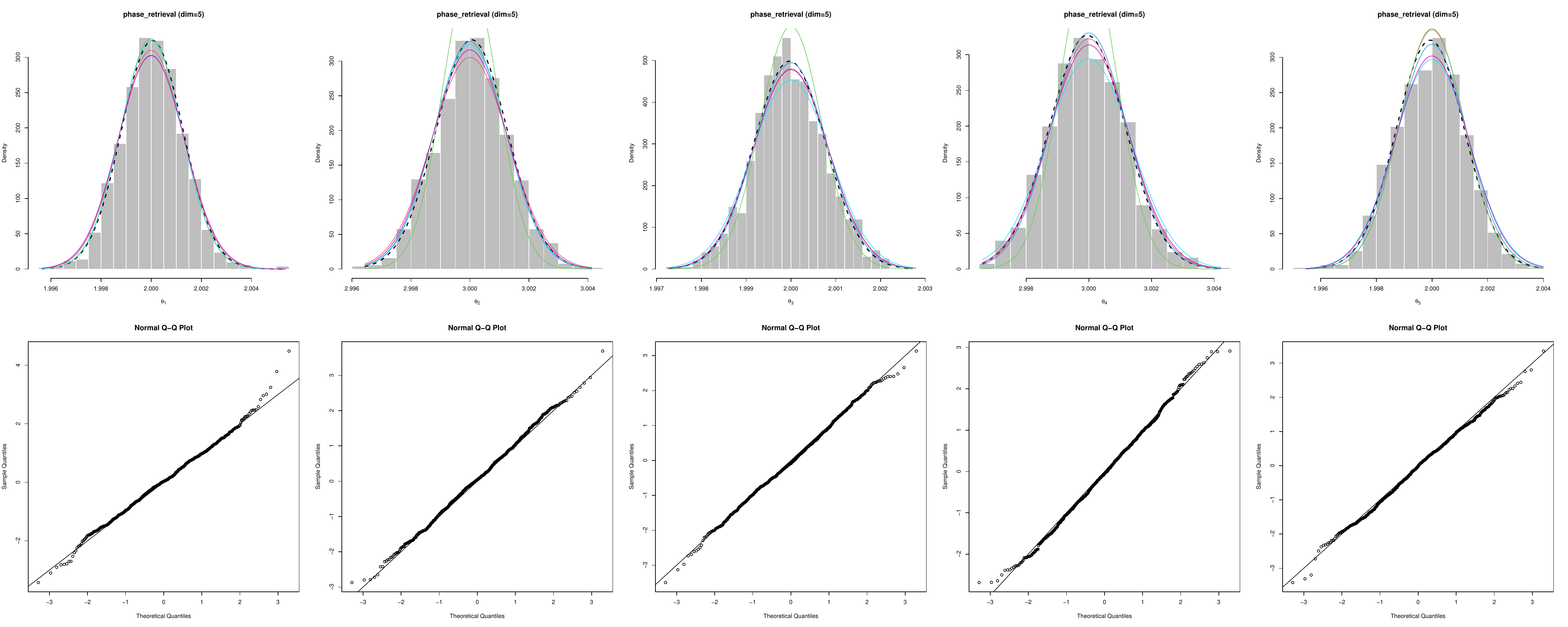}
\includegraphics[width=0.69\linewidth]{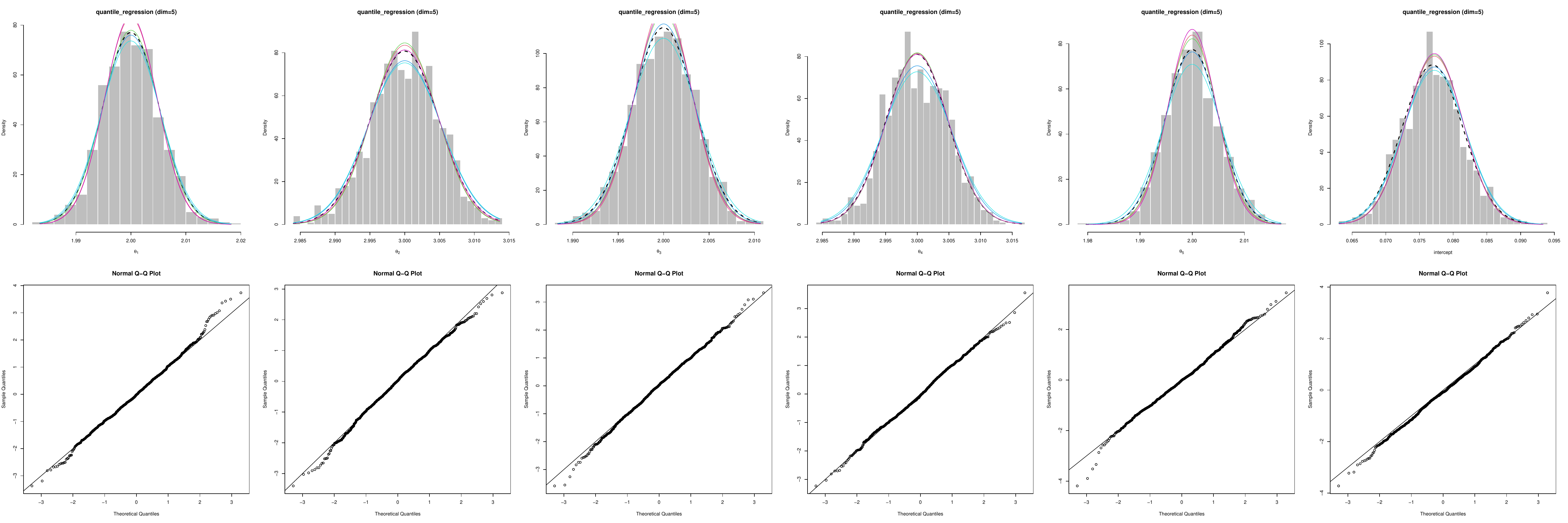}
\includegraphics[width=0.69\linewidth]{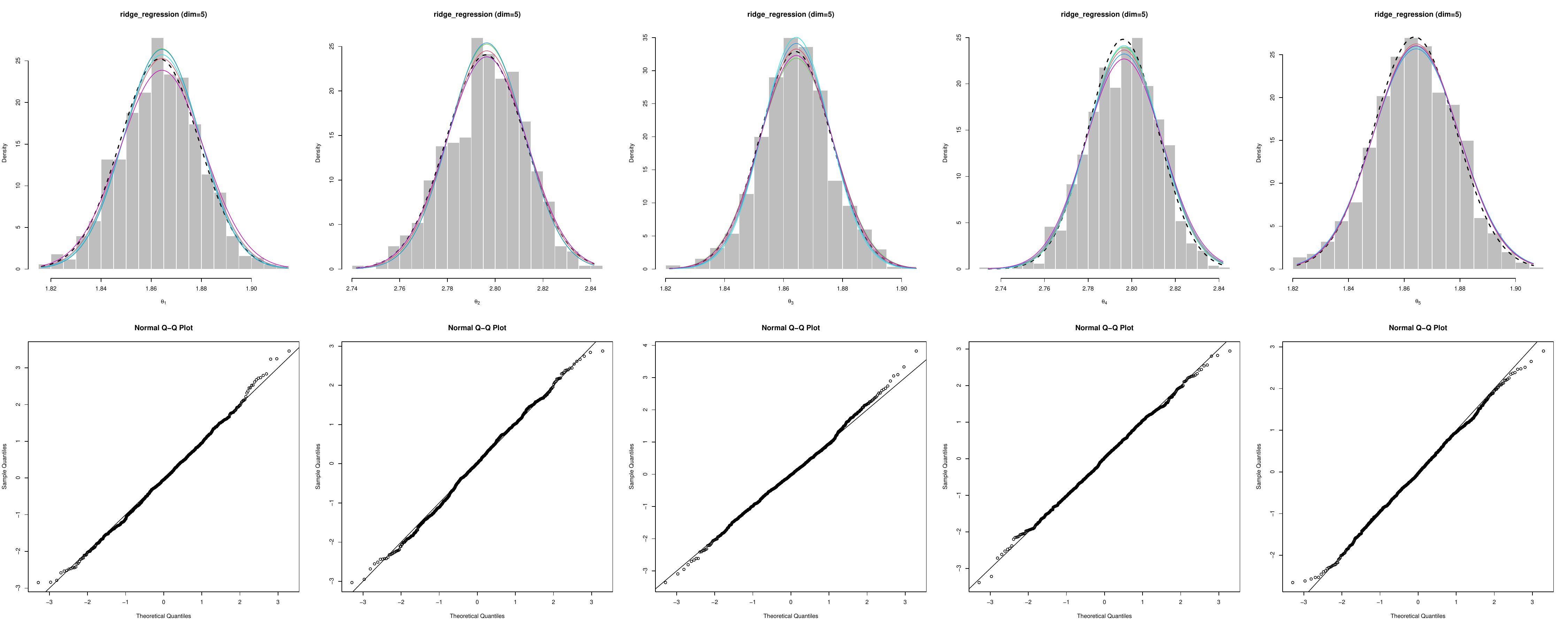}
\caption{\small Stationary points of gradient flows solved by Algorithm~\ref{alg:Euler} in linear regression, logistic regression, phase retrieval, quantile regression, and ridge regression (from top to bottom). Each column corresponds to a coordinate, where the dashed black line is the normal density plot with mean and variance based on Monte Carlo replications, the solid colored lines are the normal density plots with mean to be the population truth and variance to be the estimates in several Monte Carlo runs, and the Q-Q plot consists of the z-scores \eqref{eq:z-score} arising from Monte Carlo replications versus theoretical normal quantiles.}
\label{fig:Euler}
\end{figure}

\begin{figure}[!ht]
    \centering
\includegraphics[width=0.69\linewidth]{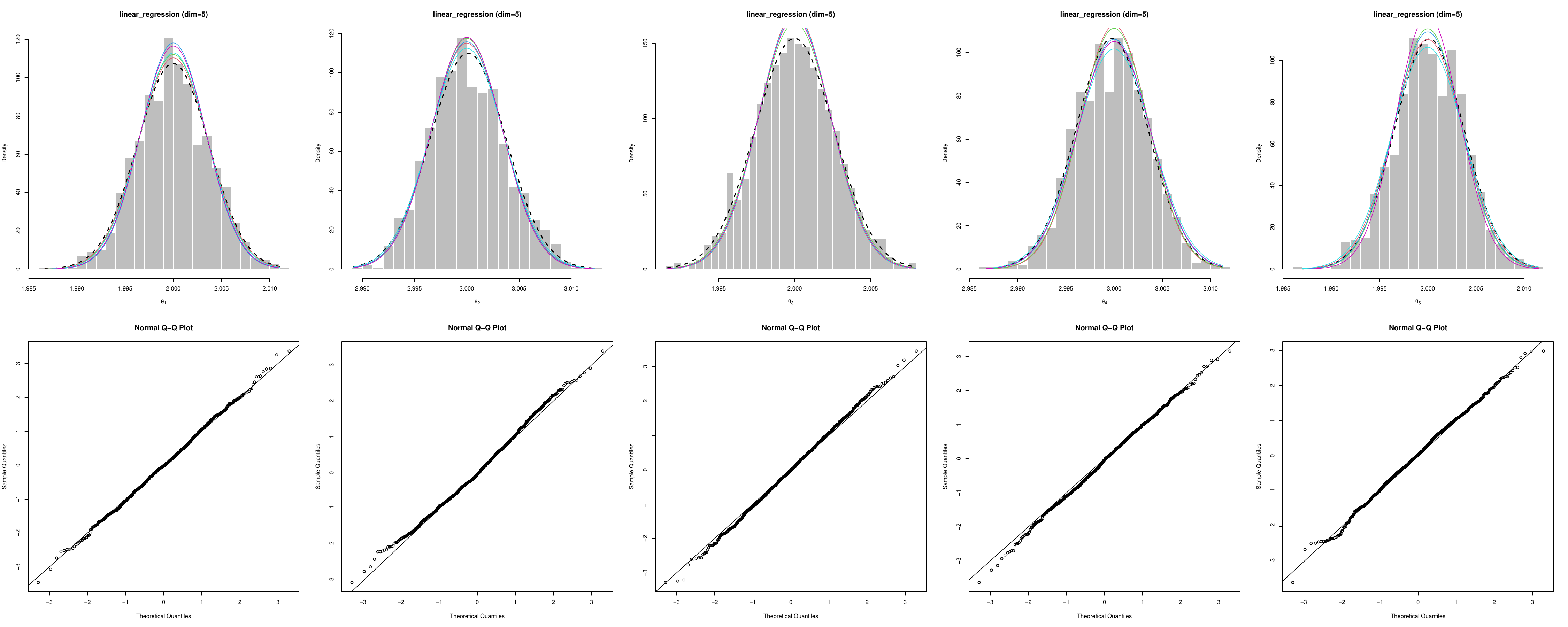}
\includegraphics[width=0.69\linewidth]{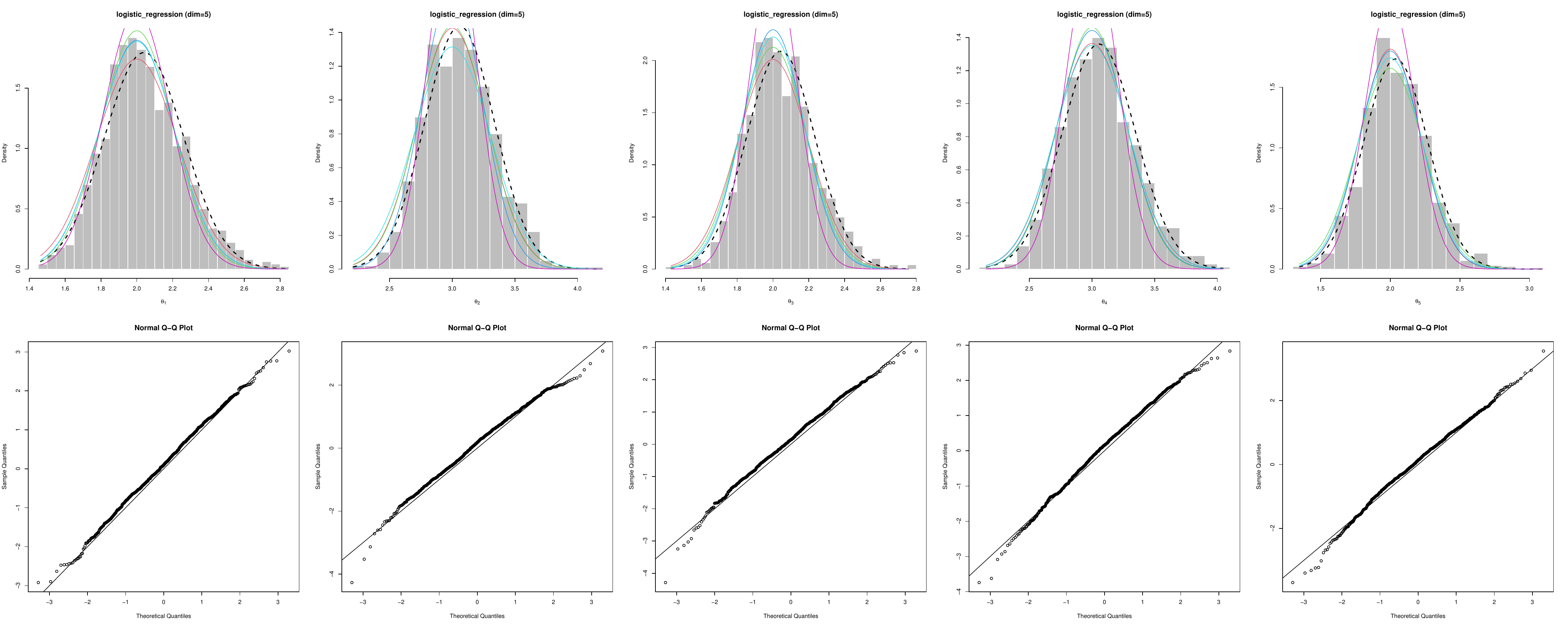}
\includegraphics[width=0.69\linewidth]{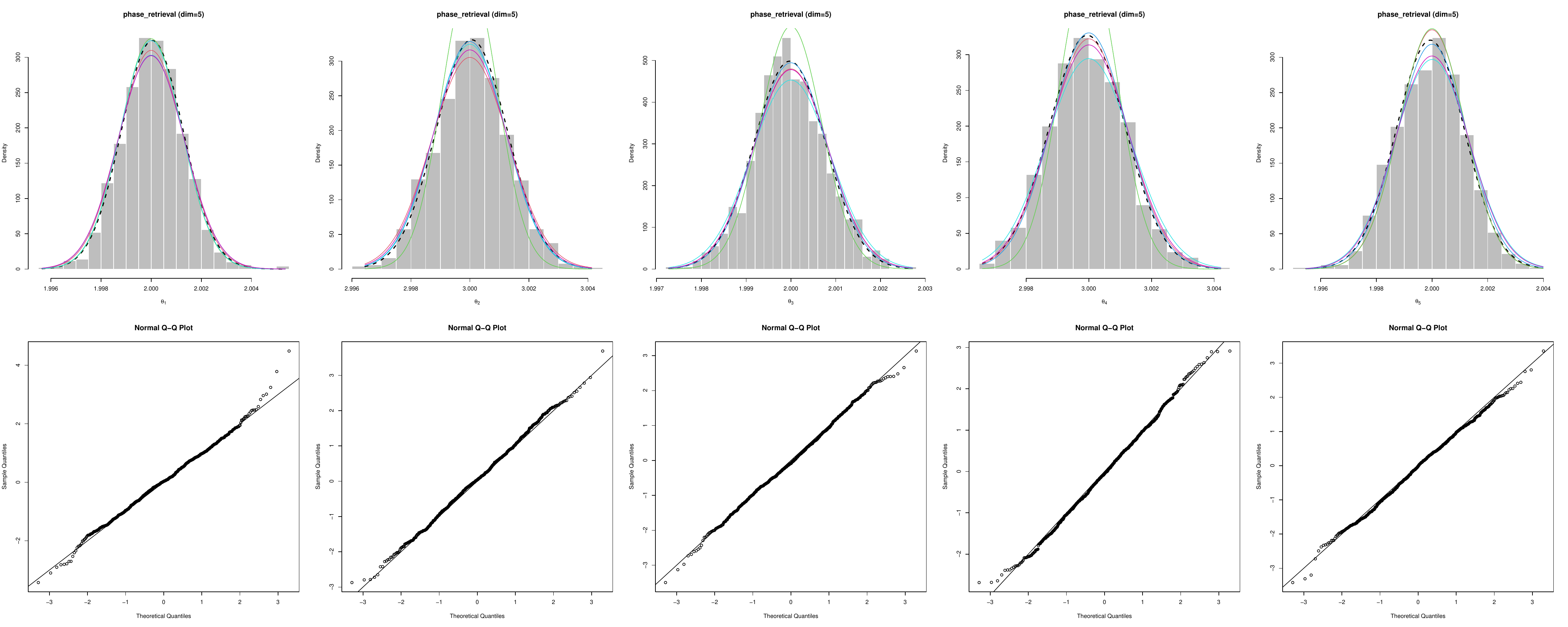}
\includegraphics[width=0.69\linewidth]{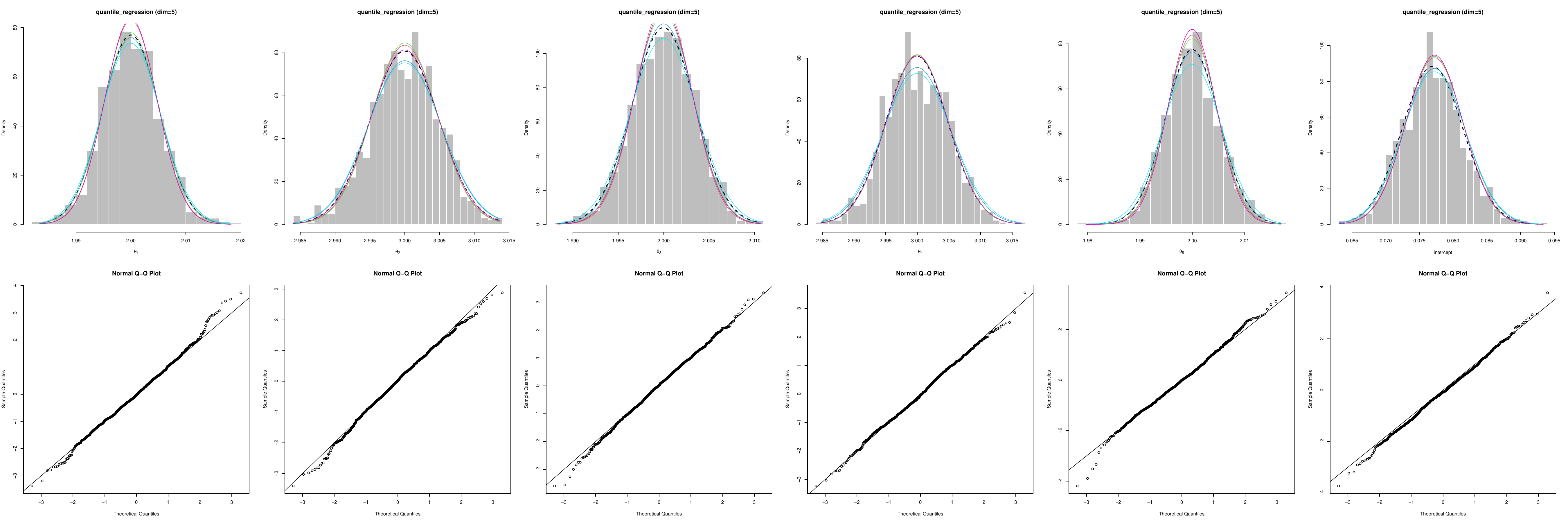}
\includegraphics[width=0.69\linewidth]{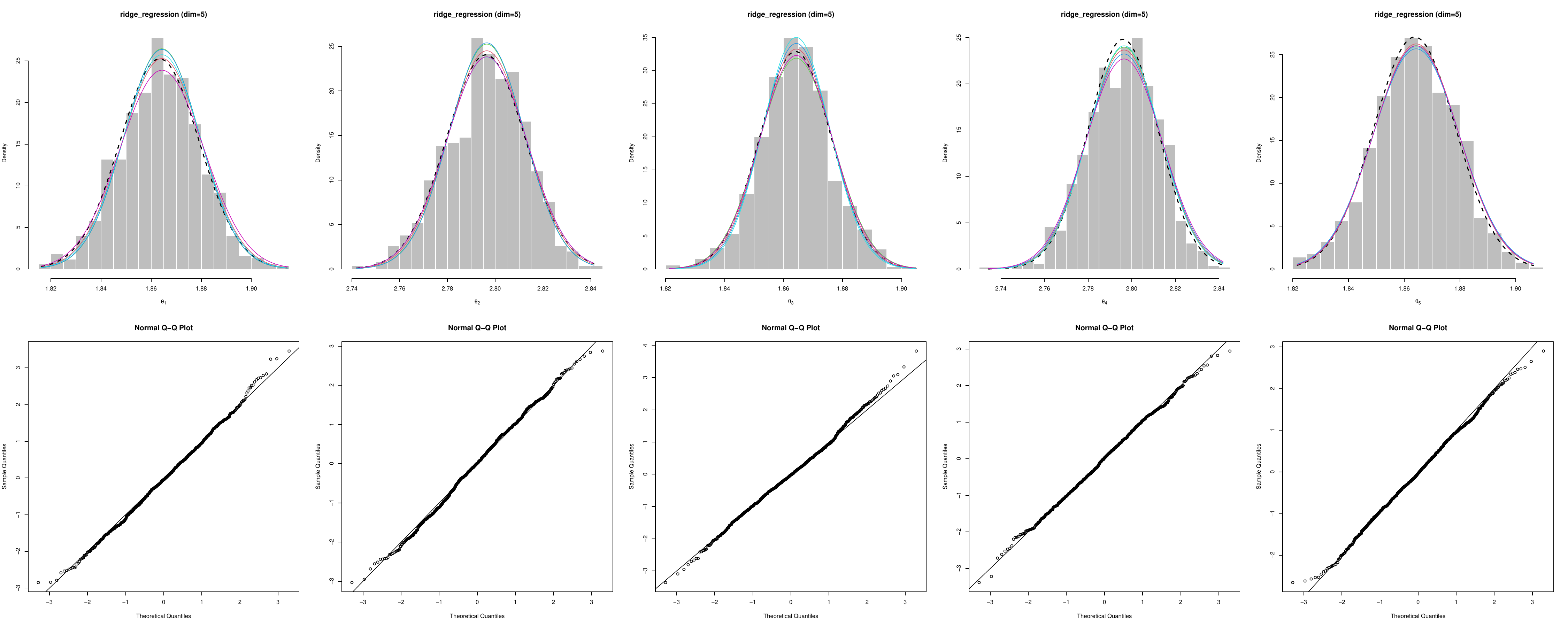}
\caption{\small Stationary points of gradient flows solved by Algorithm~\ref{alg:RK4} in linear regression, logistic regression, phase retrieval, quantile regression, and ridge regression (from top to bottom). Each column corresponds to a coordinate, where the dashed black line is the normal density plot with mean and variance based on Monte Carlo replications, the solid colored lines are the normal density plots with mean to be the population truth and variance to be the estimates in several Monte Carlo runs, and the Q-Q plot consists of the z-scores \eqref{eq:z-score} arising from Monte Carlo replications versus theoretical normal quantiles.}
\label{fig:RK4}
\end{figure}

\section*{Acknowledgments}
This research is supported by grant WBS A-0009983-00-00.

\clearpage
\forcsvlist{\renewSformat}{assumption,lemma,theorem,corollary,equation,section,table,figure,algocf}

\begin{center}
{\LARGE Supplement to ``Statistical Inference on Gradient Flows''}\\[2ex]
{\large Tongyu Li and Alexander Giessing}\\[5ex]
\end{center}

This supplementary material contains auxiliary lemmas and the proofs of all theoretical results.

\section{Additional Notation and Auxiliary Lemmas}

In a linear space of real-valued functions with (semi-)norm $\norm{\cdot}$, a bracket $[l,u]$ defined using two functions $l$ and $u$ is the set of all functions $f$ with $l \le f \le u$ (pointwise), whose size is $\norm{u-l}$. 
The bracketing number of a function class $\mathcal{F}$ with level $\varepsilon>0$ under $\norm{\cdot}$, denoted by $N_{[]}(\varepsilon,\mathcal{F},\norm{\cdot})$, is the minimum number of brackets needed to cover $\mathcal{F}$ with size not greater than $\varepsilon$ under $\norm{\cdot}$.

\begin{lemma}\label{lem:bracket}
Let $\varphi : \mathcal{T} \to \mathcal{F}$ be a surjective mapping with derivative $\dv*{\varphi(t)}{t} = \{z \mapsto \pdv*{\varphi(t)(z)}{t}\}$, where $\mathcal{T} \subset \R$ is an interval and $\mathcal{F}$ is contained in a Banach space of real-valued functions endowed with monotonic norm $\norm{\cdot}$ in the sense that $\norm{f}\le\norm{g}$ whenever $\abs{f}\le\abs{g}$. 
Then for every $\varepsilon>0$, the bracketing number of $\mathcal{F}$ under $\norm{\cdot}$ has the upper bound 
\[ N_{[]}(\varepsilon,\mathcal{F},\norm{\cdot}) \le \lceil 2\ell(\varphi)/\varepsilon \rceil ,\]
where $\ell(\varphi) = \int_\mathcal{T} \norm{\dv*{\varphi(t)}{t}} \dd{t}$ is the arc length of $\varphi$.
\end{lemma}

A function class $\mathcal{F}$ is said to have a function $F$ as its envelope if $\abs{f} \le F$ for all $f\in\mathcal{F}$.
\begin{lemma}\label{lem:product-bracket}
Let $\mathcal{F}$ and $\mathcal{G}$ be function classes with envelopes $F$ and $G$, respectively. Then for every $\varepsilon>0$ and $1/r = 1/p + 1/q$, the bracketing number of $\mathcal{F}\cdot\mathcal{G} = \{fg : f\in\mathcal{F},\,g\in\mathcal{G}\}$ satisfies 
\[ N_{[]}(2\varepsilon,\mathcal{F}\cdot\mathcal{G},L^r(P)) \le N_{[]}(\varepsilon/\norm{G}_{L^q(P)},\mathcal{F},L^p(P)) \, N_{[]}(\varepsilon/\norm{F}_{L^p(P)},\mathcal{G},L^q(P)) .\]
\end{lemma}

The covering number of a (pseudo-)metric space $(\mathcal{T},\rho)$ with level $\varepsilon>0$, denoted by $N(\varepsilon,\mathcal{T},\rho)$, is the minimum number of balls $\mathcal{B}(t,\varepsilon) = \{s : \rho(s,t) \le \varepsilon\}$ of radius $\varepsilon$ needed to cover $\mathcal{T}$.
It is convenient to write $N(\varepsilon,\mathcal{T},\rho) = N(\varepsilon,\mathcal{T},\norm{\cdot})$ if $\rho$ is induced by a (semi-)norm $\norm{\cdot}$. 
\begin{lemma}\label{lem:cover}
Let $\varphi : \mathcal{T} \to E$ be a differentiable mapping, where $\mathcal{T} \subset \R$ is an interval and $E$ is a Banach space endowed with norm $\norm{\cdot}$. Define the induced pseudometric $\rho$ on $\mathcal{T}$ by 
\[ \rho(s,t) = \norm{\varphi(s)-\varphi(t)} .\]
Then for every $\varepsilon>0$, the covering number of $(\mathcal{T},\rho)$ has the upper bound 
\[ N(\varepsilon/2,\mathcal{T},\rho) \le \lceil \ell(\varphi)/\varepsilon \rceil ,\]
where $\ell(\varphi) = \int_\mathcal{T} \norm{\dv*{\varphi(t)}{t}} \dd{t}$ is the arc length of $\varphi$.
\end{lemma}

We also need the following lemma to bound entropy integrals. 
\begin{lemma}\label{lem:ent_int}
For any $a,b > 0$, 
\[ \int_0^a \log^{1/2}(1+b/\varepsilon) \dd{\varepsilon} \le 2 a\log^{1/2}(1+b/a) .\]
\end{lemma}

\section{Proofs of Theorems}

\begin{proof}[Proof of Theorem~\ref{thm:linearization}]
Denote $\norm{\theta}_T = \norm{\theta}_{L^\infty([0,T])}$ for any $\theta : \R_{\ge0} \to \R^d$ and $T\in\R_{\ge0}$. 
By Lemma~\ref{lem:decomposition}, 
\[\begin{aligned}
\norm{\hat{\theta}-\theta^\circ-\Delta_n}_T 
&= \sup_{t\in[0,T]}\norm{\hat{\theta}(t)-\theta^\circ(t)-\Delta_n(t)}_{\ell^2} \\
&\le \sup_{t\in[0,T]}\int_0^t \normop{\Pi(t,s)} \left\{\norm{R_n(s)}_{\ell^2}+\norm{D_{\Psi}\big(\hat{\theta}(s),\theta^\circ(s)\big)}_{\ell^2}\right\} \dd{s} ,
\end{aligned}\]
and thus, by assumptions of Theorem~\ref{thm:linearization}, 
\begin{equation}\label{eq:bound}
\norm{\hat{\theta}-\theta^\circ-\Delta_n}_T \le A_*\eta_n \left(n^{-1/2} + \norm{\hat{\theta}-\theta^\circ}_T\right) + \frac{A_*L_n}{2} \norm{\hat{\theta}-\theta^\circ}_T^2 .
\end{equation}
Combining \eqref{eq:bound} and that 
\[\norm{\hat{\theta}-\theta^\circ-\Delta_n}_T \ge \norm{\hat{\theta}-\theta^\circ}_T - \norm{\Delta_n}_T \ge \norm{\hat{\theta}-\theta^\circ}_T - n^{-1/2}U_n ,\]
we obtain 
\[ \left(1-A_*\eta_n-\frac{A_*L_n}{2}\norm{\hat{\theta}-\theta^\circ}_T\right) \norm{\hat{\theta}-\theta^\circ}_T \le n^{-1/2}(U_n+A_*\eta_n) .\]
For any $T\in\R_{\ge0}$, if \[ F(T) = n^{1/2} \norm{\hat{\theta}-\theta^\circ}_T \Big/ (U_n+A_*\eta_n) ,\] then 
\[ \{F(T) \le 4\} \cap \Omega_n \subset \left\{1-A_*\eta_n-\frac{A_*L_n}{2}\norm{\hat{\theta}-\theta^\circ}_T \ge 1/2\right\} \subset \{F(T) \le 2\} ,\]
so \[ \Omega_n \subset \{F(T) \le 2\} \cup \{F(T) > 4\} .\]
Note that $F$ cannot exceed $4$ starting from below $2$ without passing through $(2,4)$, since $F$ is continuous!\footnote{The same argument was presented in Terence Tao's expository article: \url{https://www.tricki.org/article/A_non-trivial_circular_argument_can_often_be_usefully_perturbed_to_a_non-circular_one}}
\begin{center}
\begin{tikzpicture}[scale=1]
    \draw[->] (0,0) -- (5,0) node[right] {$T$};
    \draw[->] (0,0) -- (0,5);
    \draw[dashed] (0,2) -- (5,2);
    \draw[dashed,red] (0,3) -- (5,3);
    \draw[dashed] (0,4) -- (5,4);
    \fill (0,0) circle (2pt);
    \node[below left] at (0,0) {$0$};
    \draw[thick,blue]
        (0,0)
        .. controls (1,1.5) and (1.5,2.5) ..
        (2,3.2)
        .. controls (2.5,3.8) and (3.5,4.6) ..
        (4.5,4.7);
    \node at (1.8,3.5) {$F(T)$};
    \draw[very thick,red] (0,2) -- (0,4);
    \node[left,red] at (0,3) {$3$};
    \node[left] at (0,2) {$2$};
    \node[left] at (0,4) {$4$};
\end{tikzpicture}
\end{center}
Indeed, observing that $F(0) = 0$, it follows from the intermediate value theorem that for any $T > 0$, 
\[ \{F(T) > 4\} \subset \bigcup_{S\in[0,T]} \{F(S)=3\} \subset \Omega_n^\complement .\]
This shows that 
\[ \Omega_n \subset \bigcap_{T\in\R_{\ge0}} \{F(T) \le 2\} = \bigcap_{T\in\R_{\ge0}} \left\{\norm{\hat{\theta}-\theta^\circ}_T \le 2n^{-1/2}(U_n+A_*\eta_n)\right\} .\]
Plugging the above equation into \eqref{eq:bound} leads to {\normalfont(\textbf{i})}. 
To conclude, {\normalfont(\textbf{ii})} holds as a corollary of {\normalfont(\textbf{i})}.
\end{proof}

\begin{proof}[Proof of Theorem~\ref{thm:Gaussian}]
By Theorem~\ref{thm:linearization} and Slutsky's lemma \citep[Example~1.4.7]{vVW2023weak}, the convergence of $n^{1/2}\Delta_n$ suffices. Thus, Lemma~\ref{lem:lead} combined with Proposition~\ref{prop:length_bound} completes the proof. 
\end{proof}

\begin{proof}[Proof of Theorem~\ref{thm:cov}]
We derive bounds in probability for the quantities defined in Lemmas \ref{lem:est1_cov} and \ref{lem:est2_cov}.
\begin{itemize}
\item $B^\circ = \sup_t \norm{\Psi(\theta^\circ(t))}_{\ell^2} < \infty$ by Assumption~\ref{asm:converge}, since $\Psi$ is continuous.
\item $B_n = \sup_t \norm{\norm{\psi_{\theta^\circ(t)}}_{\ell^2}}_{L^2(P_n)} \le \sup_t \norm{\norm{\psi_{\theta^\circ(t)}-\psi_{\theta_*}}_{\ell^2}}_{L^2(P_n)} + \norm{\norm{\psi_{\theta_*}}_{\ell^2}}_{L^2(P_n)}$, where 
\[ \norm{\norm{\psi_{\theta_*}}_{\ell^2}}_{L^2(P_n)} = \mathcal{O}_{\Pr}\left\{\norm{\norm{\psi_{\theta_*}}_{\ell^2}}_{L^2(P)}\right\} = \mathcal{O}_{\Pr}(1) \]
by Assumption~\ref{asm:grad-Lip}, and 
\begin{align*}
\sup_t \norm{\norm{\psi_{\theta^\circ(t)}-\psi_{\theta_*}}_{\ell^2}}_{L^2(P_n)} 
&\le \lVert\dot{\psi}_{C_0}\rVert_{L^2(P_n)} \sup_t\norm{\theta^\circ(t)-\theta_*}_{\ell^2} \\
&= \mathcal{O}_\Pr\left\{\lVert\dot{\psi}_{C_0}\rVert_{L^2(P)} C_0\right\} = \mathcal{O}_\Pr(1) 
\end{align*}
by Assumptions \ref{asm:converge} and \ref{asm:grad-Lip}. 
Thus, $B_n = \mathcal{O}_\Pr(1)$.
\item $\kappa_n = \sup_t \normop{\hat{H}(\hat{\theta}(t))-H(\theta^\circ(t))} \le \sup_\theta \normop{\hat{H}(\theta)-H(\theta)} + \sup_t \normop{H(\hat{\theta}(t))-H(\theta^\circ(t))}$, where 
\[ \sup_\theta \normop{\hat{H}(\theta)-H(\theta)} = \sup_\theta \normop{(P_n-P)\ddot{m}_\theta} = \mathcal{O}_\Pr\{(n^{-1} \log n)^{1/2}\} \]
by \eqref{eq:Hess-rate}, and 
\[ \sup_t \normop{H(\hat{\theta}(t))-H(\theta^\circ(t))} \le L(r) \sup_t \norm{\hat{\theta}(t)-\theta^\circ(t)}_{\ell^2} = \mathcal{O}_\Pr(n^{-1/2}) \]
by Assumption~\ref{asm:Hess-Lip} and Theorem~\ref{thm:linearization}. 
Thus, $\kappa_n = \mathcal{O}_\Pr\{(n^{-1} \log n)^{1/2}\}$.
\item $\sigma_n = \sup_t \norm{\norm{\psi_{\hat{\theta}(t)}-\psi_{\theta^\circ(t)}}_{\ell^2}}_{L^2(P_n)}$ is not greater than 
\[ \norm{\dot{\psi}_r}_{L^2(P_n)} \sup_t\norm{\hat{\theta}(t)-\theta^\circ(t)}_{\ell^2} = \mathcal{O}_\Pr\left\{\norm{\dot{\psi}_r}_{L^2(P)}\right\} \mathcal{O}_\Pr(n^{-1/2}) = \mathcal{O}_\Pr(n^{-1/2}) ,\]
by Assumption~\ref{asm:grad-Lip} and Theorem~\ref{thm:linearization}. 
\item $\zeta_n = \sup_t \norm{(P_n-P)\psi_{\theta^\circ(t)}}_{\ell^2} = \sup_{f\in\mathcal{F}} (P_n-P)f$, with the function class \[\mathcal{F} = \{u^\top\psi_{\theta^\circ(t)} : u\in\mathbb{S}^{d-1},\, t\in\R_{\ge0}\}\] where $\mathbb{S}^{d-1}\subset\R^d$ is the unit Euclidean sphere. 
Define the (pseudo-)metrics $\rho^\circ$ on $\R_{\ge0}$ and $\rho_\oplus$ on $\mathbb{S}^{d-1}\times\R_{\ge0}$ by 
\[ \rho^\circ(s,t) = \norm{\theta^\circ(t)-\theta^\circ(s)}_{\ell^2} ,\quad \rho_\oplus\big((u,s),(v,t)) = C_0\norm{u-v}_{\ell^2} + \rho^\circ(s,t) .\]
By Assumptions \ref{asm:converge} and \ref{asm:grad-Lip}, $F = \norm{\psi_{\theta_*}}_{\ell^2} + C_0 \dot{\psi}_{C_0}$ is an envelope of $\mathcal{F}$, and 
\begin{align*}
\abs{u^\top\psi_{\theta^\circ(t)}-v^\top\psi_{\theta^\circ(s)}} &\le \norm{\psi_{\theta^\circ(t)}}_{\ell^2} \norm{u-v}_{\ell^2} + \dot{\psi}_{C_0} \norm{\theta^\circ(t)-\theta^\circ(s)}_{\ell^2} \\
&\le (F/C_0) \rho_\oplus\big((u,s),(v,t)) .
\end{align*}
Then \citet[Theorem~2.7.17]{vVW2023weak} gives 
\begin{align*}
\sup_{p\ge1} N_{[]}\big(2\varepsilon \norm{F}_{L^p(P)}/C_0 , \mathcal{F}, L^p(P)\big) &\le N(\varepsilon, \mathbb{S}^{d-1}\times\R_{\ge0}, \rho^\oplus) \\
&\le N(\varepsilon/(2C_0), \mathbb{S}^{d-1}, \norm{\cdot}_{\ell^2}) \, N(\varepsilon/2, \R_{\ge0}, \rho^\circ) .
\end{align*}
We have $N(\varepsilon/(2C_0), \mathbb{S}^{d-1}, \norm{\cdot}_{\ell^2}) \le (1+4C_0/\varepsilon)^d$ by \citet[Lemma~5.2]{vershynin2012introduction} and $N(\varepsilon/2, \R_{\ge0}, \rho^\circ) \le \lceil \ell(\theta^\circ)/\varepsilon \rceil$ by Lemma~\ref{lem:cover}, where 
\begin{align*}
\ell(\theta^\circ) = \int_0^\infty \norm{\Psi(\theta^\circ(t))}_{\ell^2} \dd{t} 
&= \int_0^\infty \norm{P\psi_{\theta^\circ(t)}-P\psi_{\theta_*}}_{\ell^2} \dd{t} \\
&\le \int_0^\infty \norm{\norm{\psi_{\theta^\circ(t)}-\psi_{\theta_*}}_{\ell^2}}_{L^2(P)} \dd{t} \\
&\le \int_0^\infty \lVert\dot{\psi}_{C_0}\rVert_{L^2(P)} C_0 \exp(-\mu t) \dd{t} \\
&= C_0 \lVert\dot{\psi}_{C_0}\rVert_{L^2(P)} /\mu < \infty
\end{align*}
by Assumptions \ref{asm:converge} and \ref{asm:grad-Lip}, which is comparable with \citet[Theorem~9]{gupta2021path}.
It follows that 
\[ \sup_{p\ge1} N_{[]}\big(2\varepsilon \norm{F}_{L^p(P)}/C_0 , \mathcal{F}, L^p(P)\big) \le (1+4C_0/\varepsilon)^d \lceil\ell(\theta^\circ)/\varepsilon\rceil .\]
Note that $\norm{F}_{L^2(P)} \le \norm{\norm{\psi_{\theta_*}}_{\ell^2}}_{L^2(P)} + C_0 \lVert\dot{\psi}_{C_0}\rVert_{L^2(P)} < \infty$ by Assumption~\ref{asm:grad-Lip}. 
Using \citet[Theorem~2.14.16]{vVW2023weak} as well as Lemma~\ref{lem:ent_int}, there is a universal constant $C>0$ such that 
\begin{align*}
n^{1/2}\E(\zeta_n) &\le C \norm{F}_{L^2(P)} \int_0^1 \big\{1 + \log N_{[]}\big(\varepsilon\norm{F}_{L^2(P)},\mathcal{F},L^2(P)\big)\big\}^{1/2} \dd{\varepsilon} \\
&\le C \norm{F}_{L^2(P)} \int_0^1 \big\{1 + d \log (1+8/\varepsilon) + \log \lceil2\ell(\theta^\circ)/C_0\varepsilon\rceil\big\}^{1/2} \dd{\varepsilon} < \infty .
\end{align*}
Thus, $\zeta_n = \mathcal{O}_\Pr(n^{-1/2})$.
\item $\gamma_n = \sup_{s,t\in\R_{\ge0}} \normop{(P_n-P)\psi_{\theta^\circ(s)}\psi_{\theta^\circ(t)}^\top} = \sup_{f\in\mathcal{F}_2} (P_n-P)f$ for $\mathcal{F}_2 = \{fg : f,g\in\mathcal{F}\}$ with envelope $F^2$. 
Note that $\norm{F}_{L^4(P)} \le \norm{\norm{\psi_{\theta_*}}_{\ell^2}}_{L^4(P)} + C_0 \lVert\dot{\psi}_{C_0}\rVert_{L^4(P)} < \infty$ by the strengthened Assumption~\ref{asm:grad-Lip}. 
Using \citet[Theorem~2.14.16]{vVW2023weak} as well as Lemmas \ref{lem:product-bracket} and \ref{lem:ent_int}, there is a universal constant $C>0$ such that 
\begin{align*}
n^{1/2}\E(\gamma_n) &\le C \norm{F^2}_{L^2(P)} \int_0^1 \big\{1 + \log N_{[]}\big(\varepsilon\norm{F^2}_{L^2(P)},\mathcal{F}_2,L^2(P)\big)\big\}^{1/2} \dd{\varepsilon} \\
&\le C \norm{F}_{L^4(P)}^2 \int_0^1 \big\{1 + 2\log N_{[]}\big(\varepsilon\norm{F}_{L^4(P)}\big/2,\mathcal{F},L^4(P)\big)\big\}^{1/2} \dd{\varepsilon} \\
&\le C \norm{F}_{L^4(P)}^2 \int_0^1 \big\{1 + 2 d \log (1+16/\varepsilon) + 2 \log \lceil4\ell(\theta^\circ)/C_0\varepsilon\rceil\big\}^{1/2} \dd{\varepsilon} < \infty .
\end{align*}
Thus, $\gamma_n = \mathcal{O}_\Pr(n^{-1/2})$.
\end{itemize}
In summary, $\{A_*\kappa_n<1\}$ has probability tending to one as $n\to\infty$, on which 
\[ c_n = A_*B_n\kappa_n/(1-A_*\kappa_n)^2 + \sigma_n/(1-A_*\kappa_n) = \mathcal{O}_\Pr\{(n^{-1} \log n)^{1/2}\} ,\]
and furthermore, 
\begin{align*}
&\sup_{t_1,t_2\in\R_{\ge0}} \normop{\hat{G}_n(t_1,t_2) - G(t_1,t_2)} \\
&\le \sup_{t_1,t_2\in\R_{\ge0}} \normop{\hat{G}_n(t_1,t_2) - G_n(t_1,t_2)} + \sup_{t_1,t_2\in\R_{\ge0}} \normop{G_n(t_1,t_2) - G(t_1,t_2)} \\
&\le (4 A_*^2 B_n c_n + 2 A_*^2 c_n^2) + A_*^2 (\gamma_n + 2 B^\circ \zeta_n + \zeta_n^2) = \mathcal{O}_\Pr\{(n^{-1} \log n)^{1/2}\} .
\end{align*}
This completes the proof.
\end{proof}

\section{Proofs of Propositions}

\begin{proof}[Proof of Proposition~\ref{prop:vanilla}]
By \eqref{eq:gradflow_emp} and \eqref{eq:gradflow_pop}, 
\[ \dv{t}\{\hat{\theta}(t)-\theta^\circ(t)\} = -\{\Psi_n(\hat{\theta}(t))-\Psi_n(\theta^\circ(t))\} -\{\Psi_n(\theta^\circ(t))-\Psi(\theta^\circ(t))\} .\]
Using the convexity of $M_n - (\lambda/2)\norm{\cdot}_{\ell^2}^2$, 
\begin{align*}
\dv{t}\norm{\hat{\theta}(t)-\theta^\circ(t)}_{\ell^2} 
&= \Big\{2\norm{\hat{\theta}(t)-\theta^\circ(t)}_{\ell^2}\Big\}^{-1} \dv{t} \norm{\hat{\theta}(t)-\theta^\circ(t)}_{\ell^2}^2 \\
&= \norm{\hat{\theta}(t)-\theta^\circ(t)}_{\ell^2}^{-1} \{\hat{\theta}(t)-\theta^\circ(t)\}^\top \dv{t} \{\hat{\theta}(t)-\theta^\circ(t)\} \\
&\le - \lambda \norm{\hat{\theta}(t)-\theta^\circ(t)}_{\ell^2} + \norm{\Psi_n(\theta^\circ(t))-\Psi(\theta^\circ(t))}_{\ell^2} .
\end{align*}
Rearrange, multiply by $\exp(\lambda t)$ and integrate to obtain the desired Duhamel bound.
Then taking the supremum bound completes the proof.
\end{proof}

\begin{proof}[Proof of Proposition~\ref{prop:length_bound}]
Write $\norm{\cdot} = \norm{\norm{\cdot}_{\ell^2}}_{L^2(P)}$ for simplicity. By \eqref{eq:ode-func_int} and \eqref{eq:func_int}, 
\[\begin{aligned}
\pdv{t}\Phi_t &= \{\psi_{\theta^\circ(t)}-\psi_{\theta_*}\} - \int_0^t H(\theta^\circ(t)) \Pi(t,s) \{\psi_{\theta^\circ(s)}-\psi_{\theta_*}\} \dd{s} \\ &\quad + \Big\{I_d - H(\theta^\circ(t)) \int_0^t \Pi(t,s) \dd{s}\Big\} \psi_{\theta_*} ,
\end{aligned}\]
and thus Minkowski's inequality leads to 
\[\begin{aligned}
\norm{\pdv{t}\Phi_t} &\le \norm{\psi_{\theta^\circ(t)}-\psi_{\theta_*}} + \int_0^t \normop{H(\theta^\circ(t))} \normop{\Pi(t,s)} \norm{\psi_{\theta^\circ(s)}-\psi_{\theta_*}} \dd{s} \\ &\quad + \normop[\Big]{I_d - H(\theta^\circ(t)) \int_0^t \Pi(t,s) \dd{s}} \norm{\psi_{\theta_*}} .
\end{aligned}\]
Integrating and using $\int_{t=0}^\infty \int_{s=0}^t \cdot \dd{s}\dd{t} = \int_{s=0}^\infty \int_{t=s}^\infty \cdot \dd{t}\dd{s}$, we have 
\[\begin{aligned}
\int_0^\infty \norm{\pdv{t}\Phi_t} \dd{t} &\le \int_0^\infty \Big\{1 + \int_s^\infty \normop{H(\theta^\circ(t))}\normop{\Pi(t,s)} \dd{t}\Big\} \norm{\psi_{\theta^\circ(s)}-\psi_{\theta_*}} \dd{s} \\ &\quad + \Big\{\int_0^\infty \normop[\Big]{I_d - H(\theta^\circ(t)) \int_0^t \Pi(t,s) \dd{s}} \dd{t}\Big\} \norm{\psi_{\theta_*}} \\
&=: \bar{\ell}_1 + \bar{\ell}_2 .
\end{aligned}\]
Lemma~\ref{lem:trans_mat_norm} implies that 
\[ \int_s^\infty \normop{H(\theta^\circ(t))}\normop{\Pi(t,s)} \dd{t} \le \Lambda \int_0^\infty A_*\lambda_* \exp(-\lambda_* r) \dd{r} = \Lambda A_* .\]
Then by Assumptions \ref{asm:converge} and \ref{asm:grad-Lip}, 
\[ \bar{\ell}_1 \le \int_0^\infty (1 + \Lambda A_*) \norm{\dot{\psi}_{C_0}}_{L^2(P)} C_0 \exp(-\mu s) \dd{s} 
= (1 + \Lambda A_*) \norm{\dot{\psi}_{C_0}}_{L^2(P)} C_0 /\mu .\]
It remains to show that 
\[ \bar{\ell}_2 / B_* = \int_0^\infty \normop[\Big]{I_d - H(\theta^\circ(t)) \int_0^t \Pi(t,s) \dd{s}} \dd{t} \le (1 + \Lambda A_*) L(C_0) C_0 /(\lambda_* \mu) + 1 /\lambda_* .\]
Introduce $\Pi_*(t,s) = \exp\{-(t-s)H(\theta_*)\}$ for $s,t\in\R_{\ge0}$. 
Direct calculation gives 
\begin{align*}
& I_d - H(\theta^\circ(t)) \int_0^t \Pi(t,s) \dd{s} + H(\theta^\circ(t)) \int_0^t \{\Pi(t,s)-\Pi_*(t,s)\} \dd{s} \\
&= I_d - H(\theta^\circ(t)) \int_0^t \Pi_*(t,s) \dd{s} \\
&= I_d - H(\theta_*) \int_0^t \Pi_*(t,s) \dd{s} - \{H(\theta^\circ(t))-H(\theta_*)\} \int_0^t \Pi_*(t,s) \dd{s} \\
&= \Pi_*(t,0) - \{H(\theta^\circ(t))-H(\theta_*)\} \int_0^t \Pi_*(t,s) \dd{s} ,
\end{align*}
and 
\begin{align*}
\int_0^t \{\Pi(t,s)-\Pi_*(t,s)\} \dd{s} 
&= \int_0^t \Pi_*(t,s) \{\Pi_*(s,t)\Pi(t,s) - I_d\} \dd{s} \\
&= \int_{s=0}^t \Pi_*(t,s) \int_{\tau=s}^t \pdv{\tau}\{\Pi_*(s,\tau)\Pi(\tau,s)\} \dd{\tau} \dd{s} \\
&= \int_{s=0}^t \Pi_*(t,s) \int_{\tau=s}^t \Pi_*(s,\tau) \{H(\theta_*)-H(\theta^\circ(\tau))\} \Pi(\tau,s) \dd{\tau} \dd{s} \\
&= \int_{\tau=0}^t \Pi_*(t,\tau) \{H(\theta_*)-H(\theta^\circ(\tau))\} \int_{s=0}^\tau \Pi(\tau,s) \dd{s} \dd{\tau} ,
\end{align*}
so 
\begin{align*}
&\normop[\Big]{I_d - H(\theta^\circ(t)) \int_0^t \Pi(t,s) \dd{s}} \\
&\le \normop{\Pi_*(t,0)} + \normop{H(\theta^\circ(t))-H(\theta_*)} \int_0^t \normop{\Pi_*(t,s)} \dd{s}
\\ &\quad + \normop{H(\theta^\circ(t))} \int_0^t \normop{\Pi_*(t,\tau)} \normop{H(\theta^\circ(\tau))-H(\theta_*)} \Big\{\int_0^\tau \normop{\Pi(\tau,s)} \dd{s}\Big\} \dd{\tau} .
\end{align*}
By Assumptions \ref{asm:converge} and \ref{asm:Hess-Lip}, 
\[ \normop{H(\theta^\circ(t))-H(\theta_*)} \le L(C_0) C_0\exp(-\mu t) ,\quad \forall t \ge 0 .\]
Since the smallest eigenvalue of $H(\theta_*)$ is larger than or equal to $\lambda_*$ by Assumption~\ref{asm:coercive}, the spectral mapping theorem leads to 
\[ \normop{\Pi_*(t,s)} \le \exp\{-\lambda_*(t-s)\} ,\quad \forall t \ge s \ge 0 .\]
In conjunction with Lemma~\ref{lem:trans_mat_norm}, 
\begin{align*}
&\int_0^\infty \normop[\Big]{I_d - H(\theta^\circ(t)) \int_0^t \Pi(t,s) \dd{s}} \dd{t} \\
&\le \int_0^\infty \normop{\Pi_*(t,0)} \dd{t} + \int_0^\infty \normop{H(\theta^\circ(t))-H(\theta_*)} \Big\{\int_0^t \normop{\Pi_*(t,s)} \dd{s}\Big\} \dd{t}
\\ &\quad + \int_{t=0}^\infty \normop{H(\theta^\circ(t))} \int_{\tau=0}^t \normop{\Pi_*(t,\tau)} \normop{H(\theta^\circ(\tau))-H(\theta_*)} \Big\{\int_0^\tau \normop{\Pi(\tau,s)} \dd{s}\Big\} \dd{\tau} \dd{t} \\
&\le \int_0^\infty \exp(-\lambda_*t) \dd{t} + \int_0^\infty L(C_0) C_0\exp(-\mu t) \Big\{\int_0^t \exp\{-\lambda_*(t-s)\} \dd{s}\Big\} \dd{t}
\\ &\quad + \int_{t=0}^\infty \Lambda \int_{\tau=0}^t \exp\{-\lambda_*(t-\tau)\} L(C_0) C_0\exp(-\mu\tau) A_*\Big[\int_0^\tau \lambda_*\exp\{-\lambda_*(\tau-s)\} \dd{s}\Big] \dd{\tau} \dd{t} \\
&\le 1/\lambda_* + L(C_0)C_0/(\lambda_*\mu) + \Lambda A_* L(C_0)C_0 \int_{\tau=0}^\infty \Big[\int_{t=\tau}^\infty \exp\{-\lambda_*(t-\tau)\} \dd{t}\Big] \exp(-\mu\tau) \dd{\tau} \\
&= 1/\lambda_* + L(C_0)C_0/(\lambda_*\mu) + \Lambda A_* L(C_0)C_0/(\lambda_*\mu) .
\end{align*}
This completes the proof.
\end{proof}

\section{Proofs of Lemmas}

\begin{proof}[Proof of Lemma~\ref{lem:decomposition}]
By definition, 
\begin{align*}
-\Big\{\dv{t}+H(\theta^\circ(t))\Big\}\{\hat{\theta}(t)-\theta^\circ(t)\}
&= -\dv{t}\hat{\theta}(t)+\dv{t}\theta^\circ(t)-H(\theta^\circ(t))\{\hat{\theta}(t)-\theta^\circ(t)\} \\
&= \Psi_n(\hat{\theta}(t)) - \Psi(\theta^\circ(t)) - H(\theta^\circ(t))\{\hat{\theta}(t)-\theta^\circ(t)\} \\
&= (\Psi_n-\Psi)(\theta^\circ(t)) + \{(\Psi_n-\Psi)(\hat{\theta}(t))-(\Psi_n-\Psi)(\theta^\circ(t))\} \\
&\quad+\Psi(\hat{\theta}(t)) - \Psi(\theta^\circ(t)) - H(\theta^\circ(t))\{\hat{\theta}(t)-\theta^\circ(t)\} \\
&= (\Psi_n-\Psi)(\theta^\circ(t)) + R_n(t) + D_{\Psi}(\hat{\theta}(t),\theta^\circ(t)) ,
\end{align*}
so variation of constants \citet[Theorem~3.12]{teschl2012ordinary} leads to the desired result.
\end{proof}

\begin{proof}[Proof of Lemma~\ref{lem:lead}]
The covariance function is clear, since \[ \Cov\{n^{1/2}\Delta_n(t_1),n^{1/2}\Delta_n(t_2)\} = G(t_1,t_2) ,\quad t_1,t_2 \in \R_{\ge0} .\]
By the Cram\'er--Wold device \citep[Example~1.3.5]{vVW2023weak}, it suffices to show the convergence of $n^{1/2}v^\top\Delta_n$ for $v\in\R^d$. 
Note that $v^\top\Delta_n(t) = (P_n-P) (-v^\top\Phi_t)$, and that 
\[ \int_0^\infty \norm{v^\top\pdv{t}\Phi_t}_{L^2(P)} \dd{t} \le \norm{v}_{\ell^2} \int_0^\infty \norm{\norm{\pdv{t}\Phi_t}_{\ell^2}}_{L^2(P)} \dd{t} \]
by the Cauchy--Schwarz inequality.
Applying Lemma~\ref{lem:uclt} completes the proof.
\end{proof}

\begin{proof}[Proof of Lemma~\ref{lem:trans_mat_norm}]
Denoting $a^+ = \max(a,0)$, by \citet[Equation~(6.6.43)]{soderlind2024logarithmic}, 
\begin{align*}
& \normop{\Pi(t,s)} \le \exp\Big\{- \int_s^t \lambda^\circ(u) \dd{u}\Big\} \\
&\le \exp\{-\lambda_0(t-s)\} \1{\{t\le t_*\}} + \exp\{-\lambda_*(t-s)+(\lambda_*-\lambda_0)(t_*-s)^+\} \1{\{t>t_*\}} \\
&= \exp\{-\lambda_*(t-s)+(\lambda_*-\lambda_0)(\min(t,t_*)-s)^+\} \\
&\le \exp\{-\lambda_*(t-s)+(\lambda_*-\lambda_0)t_*\} .
\end{align*}
This completes the proof.
\end{proof}

\begin{proof}[Proof of Lemma~\ref{lem:ode}]
Write $\tilde{\Phi}_t = \int_{s=0}^{t} \hat{\phi}_{t,s} \dd{s}$. Then by definition, $\tilde{\Phi}_0=0_d$ and 
\[ \pdv{t}\tilde{\Phi}_t = \hat{\phi}_{t,t} + \int_{s=0}^{t} \pdv{t}\hat{\phi}_{t,s} \dd{s} = \psi_{\hat{\theta}(t)} - \hat{H}(\hat{\theta}(t)) \tilde{\Phi}_t .\]
By the uniqueness of a solution to \eqref{eq:ode-empirical}, we obtain $\hat{\Phi}_t = \tilde{\Phi}_t$ as desired. 
\end{proof}

\begin{proof}[Proof of Lemma~\ref{lem:est1_cov}]
The transition matrix in \eqref{eq:transition_matrix} satisfies that $\Pi(s,t) = \Pi(t,s)^{-1}$, so 
\[ \pdv{t}\Pi(s,t) = \Pi(s,t)\Big\{\pdv{t}\Pi(t,s)\Big\}\Pi(s,t) = \Pi(s,t)H(\theta^\circ(t))\Pi(t,s)\Pi(s,t) = \Pi(s,t)H(\theta^\circ(t)) .\]
For any $s\le t$, denoting $\phi_{t,s} = \Pi(t,s) \psi_{\theta^\circ(s)}$, we have 
\begin{align*}
\hat{\phi}_{t,s}-\phi_{t,s} 
&= \Pi(t,s)\{\Pi(s,t)\hat{\phi}_{t,s}-\psi_{\hat{\theta}(s)}\}  + \Pi(t,s) \{\psi_{\hat{\theta}(s)}-\psi_{\theta^\circ(s)}\} \\
&= \Pi(t,s)\int_s^t \pdv{\tau}\{\Pi(s,\tau)\hat{\phi}_{\tau,s}\} \dd{\tau} + \Pi(t,s) \{\psi_{\hat{\theta}(s)}-\psi_{\theta^\circ(s)}\} \\
&= \Pi(t,s)\int_s^t \Pi(s,\tau)\{H(\theta^\circ(\tau))-\hat{H}(\hat{\theta}(\tau))\}\hat{\phi}_{\tau,s} \dd{\tau} + \Pi(t,s) \{\psi_{\hat{\theta}(s)}-\psi_{\theta^\circ(s)}\} \\
&= \int_s^t \Pi(t,\tau)\{H(\theta^\circ(\tau))-\hat{H}(\hat{\theta}(\tau))\}(\hat{\phi}_{\tau,s}-\phi_{\tau,s}) \dd{\tau} \\ &\quad + \int_s^t \Pi(t,\tau)\{H(\theta^\circ(\tau))-\hat{H}(\hat{\theta}(\tau))\}\Pi(\tau,s)\psi_{\theta^\circ(s)} \dd{\tau} + \Pi(t,s) \{\psi_{\hat{\theta}(s)}-\psi_{\theta^\circ(s)}\} ,
\end{align*}
and thus 
\begin{align*}
\norm{\hat{\phi}_{t,s}-\phi_{t,s}}_{\ell^2} 
&\le \int_s^t \normop{\Pi(t,\tau)} \kappa_n \norm{\hat{\phi}_{\tau,s}-\phi_{\tau,s}}_{\ell^2} \dd{\tau} \\ &\quad + \int_s^t \normop{\Pi(t,\tau)}\kappa_n\normop{\Pi(\tau,s)}\norm{\psi_{\theta^\circ(s)}}_{\ell^2} \dd{\tau} + \normop{\Pi(t,s)} \norm{\psi_{\hat{\theta}(s)}-\psi_{\theta^\circ(s)}}_{\ell^2} .
\end{align*}
By Lemma~\ref{lem:trans_mat_norm}, it follows that $u(t) = \exp(\lambda_*t)\norm{\hat{\phi}_{t,s}-\phi_{t,s}}_{\ell^2}$ satisfies 
\begin{align*}
u(t) 
&\le \int_s^t A_*\lambda_* \kappa_n \exp(\lambda_*\tau)\norm{\hat{\phi}_{\tau,s}-\phi_{\tau,s}}_{\ell^2} \dd{\tau} \\ &\quad + \int_s^t (A_*\lambda_*)^2\kappa_n\exp(\lambda_*s)\norm{\psi_{\theta^\circ(s)}}_{\ell^2} \dd{\tau} + A_*\lambda_* \exp(\lambda_*s) \norm{\psi_{\hat{\theta}(s)}-\psi_{\theta^\circ(s)}}_{\ell^2} \\
&= \int_s^t A_*\kappa_n\lambda_* u(\tau) \dd{\tau} + A_*\lambda_* \exp(\lambda_*s) \left\{A_*\norm{\psi_{\theta^\circ(s)}}_{\ell^2}\kappa_n\lambda_*(t-s) + \norm{\psi_{\hat{\theta}(s)}-\psi_{\theta^\circ(s)}}_{\ell^2}\right\} .
\end{align*}
Then Gr\"{o}nwall's inequality \citep[Lemma~2.7]{teschl2012ordinary} implies that 
\begin{align*}
& \exp(\lambda_*t)\norm{\hat{\phi}_{t,s}-\phi_{t,s}}_{\ell^2} = u(t) \\
&\le A_*\lambda_* \exp(\lambda_*s) \left\{A_*\norm{\psi_{\theta^\circ(s)}}_{\ell^2}\kappa_n\lambda_*(t-s) + \norm{\psi_{\hat{\theta}(s)}-\psi_{\theta^\circ(s)}}_{\ell^2}\right\} \exp\{A_*\kappa_n\lambda_*(t-s)\} .
\end{align*}
Rearranging the above equation and applying Minkowski's inequality, we get 
\begin{equation}\label{eq:est_phi}
\norm{\norm{\hat{\phi}_{t,s}-\phi_{t,s}}_{\ell^2}}_{L^2(P_n)} \le A_*\lambda_* \left\{A_*B_n\kappa_n\lambda_*(t-s) + \sigma_n\right\} \exp\{-(1-A_*\kappa_n)\lambda_*(t-s)\} .
\end{equation}
Also note that for any $s\le t$, by Lemma~\ref{lem:trans_mat_norm}, 
\begin{equation}\label{eq:norm_phi}
\norm{\norm{\phi_{t,s}}_{\ell^2}}_{L^2(P_n)} \le \normop{\Pi(t,s)} \norm{\norm{\psi_{\theta^\circ(s)}}_{\ell^2}}_{L^2(P_n)} \le A_* B_n \lambda_*\exp\{-\lambda_*(t-s)\} .
\end{equation}
By Weyl's inequality and the Cauchy--Schwarz inequality, 
\begin{align*}
& \normop{P_n\hat{\phi}_{t_1,s_1}\hat{\phi}_{t_2,s_2}^\top - P_n\phi_{t_1,s_1}\phi_{t_2,s_2}^\top} \le P_n\normop{\hat{\phi}_{t_1,s_1}\hat{\phi}_{t_2,s_2}^\top - \phi_{t_1,s_1}\phi_{t_2,s_2}^\top} \\
&\le P_n\normop{\phi_{t_1,s_1}(\hat{\phi}_{t_2,s_2}-\phi_{t_2,s_2})^\top} + P_n\normop{(\hat{\phi}_{t_1,s_1}- \phi_{t_1,s_1})\phi_{t_2,s_2}^\top} \\ &\quad + P_n\normop{(\hat{\phi}_{t_1,s_1}-\phi_{t_1,s_1})(\hat{\phi}_{t_2,s_2}-\phi_{t_2,s_2})^\top} \\
&\le P_n\norm{\phi_{t_1,s_1}}_{\ell^2}\norm{\hat{\phi}_{t_2,s_2}-\phi_{t_2,s_2}}_{\ell^2} + P_n\norm{\hat{\phi}_{t_1,s_1}-\phi_{t_1,s_1}}_{\ell^2}\norm{\phi_{t_2,s_2}}_{\ell^2} \\ &\quad + P_n\norm{\hat{\phi}_{t_1,s_1}-\phi_{t_1,s_1}}_{\ell^2}\norm{\hat{\phi}_{t_2,s_2}-\phi_{t_2,s_2}}_{\ell^2} \\
&\le \norm{\norm{\phi_{t_1,s_1}}_{\ell^2}}_{L^2(P_n)} \norm{\norm{\hat{\phi}_{t_2,s_2}-\phi_{t_2,s_2}}_{\ell^2}}_{L^2(P_n)} \\ &\quad + \norm{\norm{\hat{\phi}_{t_1,s_1}-\phi_{t_1,s_1}}_{\ell^2}}_{L^2(P_n)} \norm{\norm{\phi_{t_2,s_2}}_{\ell^2}}_{L^2(P_n)} \\ &\quad + \norm{\norm{\hat{\phi}_{t_1,s_1}-\phi_{t_1,s_1}}_{\ell^2}}_{L^2(P_n)} \norm{\norm{\hat{\phi}_{t_2,s_2}-\phi_{t_2,s_2}}_{\ell^2}}_{L^2(P_n)} ,
\end{align*}
and similarly, in conjunction with Jensen's inequality, 
\begin{align*}
& \normop{P_n\hat{\phi}_{t_1,s_1}P_n\hat{\phi}_{t_2,s_2}^\top - P_n\phi_{t_1,s_1}P_n\phi_{t_2,s_2}^\top} \\
&\le \normop{P_n(\hat{\phi}_{t_1,s_1}-\phi_{t_1,s_1})P_n(\hat{\phi}_{t_2,s_2}-\phi_{t_2,s_2})^\top} \\ &\quad + \normop{P_n\phi_{t_1,s_1}P_n(\hat{\phi}_{t_2,s_2}-\phi_{t_2,s_2})^\top} + \normop{P_n(\hat{\phi}_{t_1,s_1}- \phi_{t_1,s_1})P_n\phi_{t_2,s_2}^\top} \\
&\le \norm{P_n(\hat{\phi}_{t_1,s_1}-\phi_{t_1,s_1})}_{\ell^2}\norm{P_n(\hat{\phi}_{t_2,s_2}-\phi_{t_2,s_2})}_{\ell^2} \\ &\quad + \norm{P_n\phi_{t_1,s_1}}_{\ell^2} \norm{P_n(\hat{\phi}_{t_2,s_2}-\phi_{t_2,s_2})}_{\ell^2} + \norm{P_n(\hat{\phi}_{t_1,s_1}-\phi_{t_1,s_1})}_{\ell^2} \norm{P_n\phi_{t_2,s_2}}_{\ell^2} \\
&\le \norm{\norm{\hat{\phi}_{t_1,s_1}-\phi_{t_1,s_1}}_{\ell^2}}_{L^2(P_n)} \norm{\norm{\hat{\phi}_{t_2,s_2}-\phi_{t_2,s_2}}_{\ell^2}}_{L^2(P_n)} \\ &\quad + \norm{\norm{\phi_{t_1,s_1}}_{\ell^2}}_{L^2(P_n)} \norm{\norm{\hat{\phi}_{t_2,s_2}-\phi_{t_2,s_2}}_{\ell^2}}_{L^2(P_n)} \\ &\quad + \norm{\norm{\hat{\phi}_{t_1,s_1}-\phi_{t_1,s_1}}_{\ell^2}}_{L^2(P_n)} \norm{\norm{\phi_{t_2,s_2}}_{\ell^2}}_{L^2(P_n)} .
\end{align*}
Plugging \eqref{eq:est_phi} and \eqref{eq:norm_phi} into the above inequalities, we can conclude that for any $s_1\le t_1$ and $s_2\le t_2$, 
\begin{align*}
& \max(\normop{P_n\hat{\phi}_{t_1,s_1}\hat{\phi}_{t_2,s_2}^\top - P_n\phi_{t_1,s_1}\phi_{t_2,s_2}^\top} , \normop{P_n\hat{\phi}_{t_1,s_1}P_n\hat{\phi}_{t_2,s_2}^\top - P_n\phi_{t_1,s_1}P_n\phi_{t_2,s_2}^\top}) \\
&\le (A_*\lambda_*)^2 B_n \sum_{(j,j')=(1,2),(2,1)} \{A_*B_n\kappa_n\lambda_*(t_j-s_j) + \sigma_n\} \\ &\hspace{12em}\times \exp\{-(1-A_*\kappa_n)\lambda_*(t_j-s_j)-\lambda_*(t_{j'}-s_{j'})\} \\ &\quad + (A_*\lambda_*)^2 \prod_{j=1,2} \{A_*B_n\kappa_n\lambda_*(t_j-s_j) + \sigma_n\} \exp\{-(1-A_*\kappa_n)\lambda_*(t_j-s_j)\} .
\end{align*}
Therefore, on the event $\{A_*\kappa_n<1\}$, for any $t_1,t_2\in\R_{\ge0}$, by \eqref{eq:func_int} and Lemma~\ref{lem:ode}, 
\begin{align*}
&\quad \normop{\hat{G}_n(t_1,t_2) - G_n(t_1,t_2)} \\
&= \bigg\lVert\int_{s_1=0}^{t_1}\int_{s_2=0}^{t_2} \Big\{\big(P_n\hat{\phi}_{t_1,s_1}\hat{\phi}_{t_2,s_2}^\top - P_n\phi_{t_1,s_1}P_n\phi_{t_2,s_2}^\top\big) \\ &\hspace{8em} - \big(P_n\phi_{t_1,s_1}\phi_{t_2,s_2}^\top - P_n\hat{\phi}_{t_1,s_1}P_n\hat{\phi}_{t_2,s_2}^\top\big)\Big\} \dd{s_1}\dd{s_2}\bigg\rVert_{\mathrm{op}} \\
&\le \int_{s_1=0}^{t_1}\int_{s_2=0}^{t_2} \Big(\normop{P_n\hat{\phi}_{t_1,s_1}\hat{\phi}_{t_2,s_2}^\top - P_n\phi_{t_1,s_1}\phi_{t_2,s_2}^\top} \\ &\hspace{7em} + \normop{P_n\phi_{t_1,s_1}P_n\phi_{t_2,s_2}^\top - P_n\hat{\phi}_{t_1,s_1}P_n\hat{\phi}_{t_2,s_2}^\top}\Big) \dd{s_1}\dd{s_2} \\
&\le 2 (A_*\lambda_*)^2 B_n \sum_{(j,j')=(1,2),(2,1)} \int_{s_{j'}=0}^{t_{j'}} \exp\{-\lambda_*(t_{j'}-s_{j'})\} \dd{s_{j'}} \\ &\hspace{9em}\times \int_{s_j=0}^{t_j} \{A_*B_n\kappa_n\lambda_*(t_j-s_j) + \sigma_n\} \exp\{-(1-A_*\kappa_n)\lambda_*(t_j-s_j)\} \dd{s_j} \\ &\quad + 2 (A_*\lambda_*)^2 \prod_{j=1,2} \int_{s_j=0}^{t_j} \{A_*B_n\kappa_n\lambda_*(t_j-s_j) + \sigma_n\} \exp\{-(1-A_*\kappa_n)\lambda_*(t_j-s_j)\} \dd{s_j} \\
&\le 4 A_*^2 B_n \{A_*B_n\kappa_n/(1-A_*\kappa_n)^2 + \sigma_n/(1-A_*\kappa_n)\} \\ &\quad + 2 A_*^2 \{A_*B_n\kappa_n/(1-A_*\kappa_n)^2 + \sigma_n/(1-A_*\kappa_n)\}^2 .
\end{align*}
The proof is now completed.
\end{proof}

\begin{proof}[Proof of Lemma~\ref{lem:est2_cov}]
Let $\phi_{t,s} = \Pi(t,s) \psi_{\theta^\circ(s)}$ for $s\le t$. We have 
\begin{align*}
& \normop{\Cov_{P_n}(\phi_{t_1,s_1},\phi_{t_2,s_2}) - \Cov_P(\phi_{t_1,s_1},\phi_{t_2,s_2})} \\
&= \normop{\Pi(t_1,s_1) [\Cov_{P_n}\{\psi_{\theta^\circ(s_1)},\psi_{\theta^\circ(s_2)}\} - \Cov_P\{\psi_{\theta^\circ(s_1)},\psi_{\theta^\circ(s_2)}\}] \Pi(t_2,s_2)^\top} \\
&\le \normop{\Pi(t_1,s_1)}\normop{\Pi(t_2,s_2)} \normop{\Cov_{P_n}\{\psi_{\theta^\circ(s_1)},\psi_{\theta^\circ(s_2)}\} - \Cov_P\{\psi_{\theta^\circ(s_1)},\psi_{\theta^\circ(s_2)}\}} .
\end{align*}
and
\begin{align*}
& \normop{\Cov_{P_n}\{\psi_{\theta^\circ(s_1)},\psi_{\theta^\circ(s_2)}\} - \Cov_P\{\psi_{\theta^\circ(s_1)},\psi_{\theta^\circ(s_2)}\}} \\
&\le \normop{P_n\psi_{\theta^\circ(s_1)}\psi_{\theta^\circ(s_2)}^\top - P\psi_{\theta^\circ(s_1)}\psi_{\theta^\circ(s_2)}^\top} + \normop{P_n\psi_{\theta^\circ(s_1)}P_n\psi_{\theta^\circ(s_2)}^\top - P\psi_{\theta^\circ(s_1)}P\psi_{\theta^\circ(s_2)}^\top} \\
&\le \gamma_n + \normop{(P_n-P)\psi_{\theta^\circ(s_1)}(P_n-P)\psi_{\theta^\circ(s_2)}^\top} \\ &\quad + \normop{P\psi_{\theta^\circ(s_1)}(P_n-P)\psi_{\theta^\circ(s_2)}^\top} + \normop{(P_n-P)\psi_{\theta^\circ(s_1)}P\psi_{\theta^\circ(s_2)}^\top} \\
&\le \gamma_n + \norm{(P_n-P)\psi_{\theta^\circ(s_1)}}_{\ell^2}\norm{(P_n-P)\psi_{\theta^\circ(s_2)}}_{\ell^2} \\ &\quad + \norm{P\psi_{\theta^\circ(s_1)}}_{\ell^2} \norm{(P_n-P)\psi_{\theta^\circ(s_2)}}_{\ell^2} + \norm{(P_n-P)\psi_{\theta^\circ(s_1)}}_{\ell^2} \norm{P\psi_{\theta^\circ(s_2)}}_{\ell^2} \\
&\le \gamma_n + \zeta_n^2 + 2 B^\circ \zeta_n .
\end{align*}
Thus, for any $t_1,t_2\in\R_{\ge0}$, by \eqref{eq:func_int} and \eqref{asm:transition_matrix}, 
\begin{align*}
&\normop{G_n(t_1,t_2) - G(t_1,t_2)} \\
&= \left\lVert\int_{s_1=0}^{t_1}\int_{s_2=0}^{t_2} \left\{\Cov_{P_n}(\phi_{t_1,s_1},\phi_{t_2,s_2}) - \Cov_P(\phi_{t_1,s_1},\phi_{t_2,s_2})\right\} \dd{s_1}\dd{s_2}\right\rVert_{\mathrm{op}} \\
&\le \int_{s_1=0}^{t_1}\int_{s_2=0}^{t_2} \normop{\Cov_{P_n}(\phi_{t_1,s_1},\phi_{t_2,s_2}) - \Cov_P(\phi_{t_1,s_1},\phi_{t_2,s_2})} \dd{s_1}\dd{s_2} \\
&\le \int_{s_1=0}^{t_1}\int_{s_2=0}^{t_2} \normop{\Pi(t_1,s_1)}\normop{\Pi(t_2,s_2)} (\gamma_n + \zeta_n^2 + 2 B^\circ \zeta_n) \dd{s_1}\dd{s_2} \\
&\le A_*^2 (\gamma_n + \zeta_n^2 + 2 B^\circ \zeta_n) .
\end{align*}
This completes the proof.
\end{proof}

\begin{proof}[Proof of Lemma~\ref{lem:uclt}]
We apply Donsker's theorem \citep[Theorem~2.5.6]{vVW2023weak}, for which we need to show the finiteness of the bracketing integral corresponding to $\mathcal{F} = \{\varphi_t\}_{t\in \mathcal{T}}$ in $L^2(P)$.
If $\ell = \int_\mathcal{T} \norm{\pdv{t}\varphi_t}_{L^2(P)} \dd{t}$, then 
\[\sup_{t\in \mathcal{T}} \norm{\varphi_t}_{L^2(P)} \le \norm{\varphi_{t_0}}_{L^2(P)} + \ell < \infty\] by Minkowski's inequality, implying $\mathcal{F} \subset L^2(P)$. 
Furthermore, by Lemmas \ref{lem:bracket} and \ref{lem:ent_int}, 
\[\begin{aligned}
\int_0^\infty \log^{1/2} N_{[]}(\varepsilon,\mathcal{F},L^2(P)) \dd{\varepsilon}
&\le \int_0^\infty \log^{1/2} \lceil 2 \ell / \varepsilon \rceil \dd{\varepsilon} \\
&\le \int_0^{2\ell} \log^{1/2}(1 + 2\ell / \varepsilon) \dd{\varepsilon} \\
&\le (4 \log^{1/2}2) \ell < \infty .
\end{aligned}\]
This completes the proof.
\end{proof}

\begin{proof}[Proof of Lemma~\ref{lem:length-ode}]
Define $V_\varphi(t) = V(\varphi(t))$ and $a_\varphi = a(\sup_t\norm{\varphi(t)})$. Then 
\begin{align*}
\dv{V_\varphi(t)}{t} &= \inprod{\nabla V(\varphi(t))}{F(\varphi(t),t)} \\
&= \inprod{\nabla V(\varphi(t))}{F_*(\varphi(t))} + \inprod{\nabla V(\varphi(t))}{F(\varphi(t),t)-F_*(\varphi(t))} \\
&\le -\lambda\norm{\varphi(t)-u_*}^2 + L\norm{\varphi(t)-u_*} \cdot a_\varphi r(t) \\
&\le -\lambda C_2 V_\varphi(t) + a_\varphi L C_1^{1/2} V_\varphi^{1/2}(t) r(t) .
\end{align*}
It follows that 
\[ \dv{V_\varphi^{1/2}(t)}{t} = \frac{1}{2V_\varphi^{1/2}}\dv{V_\varphi(t)}{t}
\le -\frac{\lambda C_2}{2} V_\varphi^{1/2}(t) + \frac{a_\varphi L C_1^{1/2}}{2} r(t) .\]
Rearrange, multiply by $\exp\{(\lambda C_2 /2)(t-t_0)\}$ and integrate to obtain the Duhamel bound 
\[ V_\varphi^{1/2}(t) \le \exp\Big\{-\frac{\lambda C_2}{2}(t-t_0)\Big\} V_\varphi^{1/2}(t_0) + \frac{a_\varphi L C_1^{1/2}}{2} \int_{t_0}^t \exp\Big\{-\frac{\lambda C_2}{2}(t-s)\Big\} r(s) \dd{s} .\]
Noticing that $\int_{t=t_0}^\infty \int_{s=t_0}^t \cdot \dd{s}\dd{t} = \int_{s=t_0}^\infty \int_{t=s}^\infty \cdot \dd{t}\dd{s}$, we have 
\[ \int_{t_0}^\infty V_\varphi^{1/2}(t) \dd{t} \le \frac{2}{\lambda C_2} V_\varphi^{1/2}(t_0) + \frac{a_\varphi L C_1^{1/2}}{\lambda C_2} \int_{t_0}^\infty r(s) \dd{s} .\]
Therefore, 
\[ \int_{t_0}^\infty \norm{\varphi(t)-u_*} \dd{t} \le C_1^{1/2} \int_{t_0}^\infty V_\varphi^{1/2}(t) \dd{t} \le \frac{2C_1^{1/2}}{\lambda C_2^{3/2}} \norm{\varphi(t_0)-u_*} + \frac{a_\varphi L C_1}{\lambda C_2} \int_{t_0}^\infty r(t) \dd{t} .\]
This leads to the desired result, since 
\[ \norm{\dv{\varphi(t)}{t}} \le \norm{F(\varphi(t),t)-F_*(\varphi(t))} + \norm{F_*(\varphi(t))} \le a_\varphi r(t) + \Lambda \norm{\varphi(t)-u_*} .\]
The proof is completed.
\end{proof}

\begin{proof}[Proof of Lemma~\ref{lem:bracket}]
Define the (pointwise) total variation $V_{st} = \int_s^t \abs{\dv*{\varphi(u)}{u}} \dd{u}$ for $s\le t$, which satisfies \[\sup_{u\in[s,t]}\abs{\varphi(u)-\varphi(s)} \le V_{st}\] and by Henri Cartan's vectorial mean value theorem, \[\norm{V_{st}} \le \ell(\varphi|_{[s,t]}) = \int_s^t \norm{\dv*{\varphi(u)}{u}} \dd{u}.\]
If $t_0 = \inf \mathcal{T}$, $t_N = \sup \mathcal{T}$, and $t_j \in \mathcal{T}$, $j=1,\dots,N-1$, are chosen such that $\ell(\varphi|_{[t_{j-1},t_j]}) \le \varepsilon/2$ for $j=1,\dots,N$, then any $t \in [t_{j-1},t_j]$ satisfies that \[\varphi(t_{j-1})-V_{t_{j-1}t_j} \le \varphi(t) \le \varphi(t_{j-1})+V_{t_{j-1}t_j},\] i.e., the brackets $\varphi(t_{j-1}) \pm V_{t_{j-1}t_j}$ cover $\mathcal{F}$, and their sizes are bounded by $\norm{2V_{t_{j-1}t_j}} \le \varepsilon$. 
The smallest $N$ is exactly $\lceil 2\ell(\varphi)/\varepsilon \rceil$, completing the proof.
\end{proof}

\begin{proof}[Proof of Lemma~\ref{lem:product-bracket}]
Suppose that $\mathcal{F}$ (resp.\ $\mathcal{G}$) admits a cover of brackets $[l_j,u_j]$, $j=1,\dots,N$, s.t.\ $\norm{u_j-l_j}_{L^p(P)} \le \varepsilon/\norm{G}_{L^q(P)}$ (resp.\ $[l_j',u_j']$, $j=1,\dots,N'$, s.t.\ $\lVert u_j'-l_j'\rVert_{L^q(P)} \le \varepsilon/\norm{F}_{L^p(P)}$). 
Define (pointwise) \[L_{jk} = \min(l_j l_k', l_j u_k', u_j l_k', u_j u_k'),\] \[U_{jk} = \max(l_j l_k', l_j u_k', u_j l_k', u_j u_k').\]
It can be seen that $[l_j,u_j]\cdot[l_k',u_k'] \subset [L_{jk},U_{jk}]$ and 
\begin{align*}
U_{jk}-L_{jk} 
&\le \max_{f_1,f_2\in\{l_j,u_j\}} \max_{g_1,g_2\in\{l_k',u_k'\}} \abs{f_1g_1-f_2g_2} \\
&\le \max_{f_1,f_2\in\{l_j,u_j\}} \max_{g_1,g_2\in\{l_k',u_k'\}} \{\abs{f_1}\cdot\abs{g_1-g_2} + \abs{f_1-f_2}\cdot\abs{g_2}\} \\
&\le F (u_k'-l_k') + (u_j-l_j) G .
\end{align*}
By Minkowski's inequality and H\"older's inequality, 
\begin{align*}
\norm{U_{jk}-L_{jk}}_{L^r(P)} &\le \norm{F (u_k'-l_k')}_{L^r(P)} + \norm{(u_j-l_j) G}_{L^r(P)} \\
&\le \norm{F}_{L^p(P)} \norm{u_k'-l_k'}_{L^q(P)} + \norm{u_j-l_j}_{L^p(P)} \norm{G}_{L^q(P)} \le 2 \varepsilon .
\end{align*}
This gives the desired result.
\end{proof}

\begin{proof}[Proof of Lemma~\ref{lem:cover}]
By Henri Cartan's vectorial mean value theorem, \[\rho(s,t) \le \ell(\varphi|_{[s,t]}) = \int_s^t \norm{\dv*{\varphi(u)}{u}} \dd{u} ,\quad \forall s \le t .\]
If $t_j \in \mathcal{T}$, $j=1,\dots,N$, are chosen such that $\max\{\ell(\varphi|_{[\inf \mathcal{T}, t_1]}) , \ell(\varphi|_{[t_N, \sup \mathcal{T}]})\} \le \varepsilon/2$ and that $\ell(\varphi|_{[t_{j-1},t_j]}) \le \varepsilon$ for $j=2,\dots,N$, then any $t \in \mathcal{T}$ satisfies that $\min_j \rho(t,t_j) \le \varepsilon/2$, i.e., $t_j$'s form an $(\varepsilon/2)$-covering of $(\mathcal{T},\rho)$.
The smallest $N$ is exactly $\lceil \ell(\varphi)/\varepsilon \rceil$, completing the proof.
\end{proof}

\begin{proof}[Proof of Lemma~\ref{lem:ent_int}]
Integrating by parts and taking the change of variables $t = \log^{1/2}(1+b/\varepsilon)$, we obtain 
\[ \int_0^a \log^{1/2}(1+b/\varepsilon) \dd{\varepsilon} = a\log^{1/2}(1+b/a) + \int_{\log^{1/2}(1+b/a)}^\infty \frac{b}{\exp(t^2)-1} \dd{t} ,\]
where we have used $\lim_{\varepsilon\searrow0} \varepsilon \log^{1/2}(1+b/\varepsilon) = 0$.
Note that the function $h(x) = \{\exp(x)-1\}/x$ is increasing on $(0,\infty)$, which can be seen from that 
\[ \dv{h(x)}{x} = \frac{\exp(x)x-\{\exp(x)-1\}}{x^2} = \frac{\exp(x)}{x^2}\{x-1+\exp(-x)\} > 0 .\]
It follows that $\{\exp(t^2)-1\}/t^2 \ge (b/a)/\log(1+b/a)$, and thus 
\[ \int_{\log^{1/2}(1+b/a)}^\infty \frac{b}{\exp(t^2)-1} \dd{t} \le \int_{\log^{1/2}(1+b/a)}^\infty \frac{a\log(1+b/a)}{t^2} \dd{t} = a\log^{1/2}(1+b/a) .\]
This completes the proof.
\end{proof}

\clearpage
\phantomsection
\addcontentsline{toc}{section}{References}
\bibliographystyle{apalike}
\bibliography{ref}
\end{document}